\documentclass[11pt,reqno,oneside]{amsproc}
\title[A free boundary inviscid model of flow-structure interaction]{A free boundary inviscid model of flow-structure interaction}

\author[M.S.~Ayd{\i}n]{Mustafa Sencer Ayd{\i}n}
\address{Department of Mathematics, University of Southern California, Los Angeles, CA 90089}
\email{maydin@usc.edu}

\author[I.~Kukavica]{Igor Kukavica}
\address{Department of Mathematics, University of Southern California, Los Angeles, CA 90089}
\email{kukavica@usc.edu}

\author[A.~Tuffaha]{Amjad Tuffaha}
\address{Department of Mathematics and Statistics, American University of Sharjah, Sharjah, UAE}
\email{atufaha\char'100aus.edu}

\chardef\forshowkeys=0
\chardef\refcheck=0
\chardef\showllabel=0
\chardef\sketches=0
\chardef\showcolors=0
\chardef\showfont=0        

\usepackage{enumitem}
\usepackage{datetime}

\usepackage{fancyhdr}
\usepackage{comment}
\allowdisplaybreaks

\ifnum\forshowkeys=1

\usepackage[notref,notcite,color]{showkeys}
\fi

\usepackage[margin=1in]{geometry}
\usepackage{amsmath, amsthm, amssymb}
\usepackage{times}
\usepackage{graphicx}
\usepackage[usenames,dvipsnames,svgnames,table]{xcolor}
\usepackage{marginnote}
\usepackage[most]{tcolorbox}

\usepackage{tikz}

\ifnum\refcheck=1
\usepackage{refcheck}
\fi

\begin{document}

	\def\rr{r}
	\def\bnew{\colr NEW: }
	\def\enew{\colb {}}
	\def\bold{\colu OLD: }
	\def\eold{\colb {}}
	\def\XX{X}
	\def\YY{Y}
	\def\ZZZ{Z}
	\def\Tmax{T_\text{max}}
	\def\Cabs{C_\text{abs}}
	\def\fomega{f_{\Omega}}
	\def\fgamma{f_{\Gamma_1}}

	\def\intint{\int\!\!\!\!\int}
	\def\OO{\mathcal O}
	\def\SS{\mathbb S}
	\def\CC{\mathbb C}
	\def\RR{\mathbb R}
	\def\TT{\mathbb T}
	\def\ZZ{\mathbb Z}
	\def\HH{\mathbb H}
	\def\RSZ{\mathcal R}
	\def\LL{\mathcal L}
	\def\SL{\LL^1}
	\def\ZL{\LL^\infty}
	\def\GG{\mathcal G}
	\def\tt{\langle t\rangle}
	\def\erf{\mathrm{Erf}}
	\def\mgt#1{\textcolor{magenta}{#1}}
	\def\ff{\rho}
	\def\gg{G}
	\def\sqrtnu{\sqrt{\nu}}
	\def\ww{w}
	\def\ft#1{#1_\xi}
	\def\lec{\lesssim}
	\def\ges{\gtrsim}
	\renewcommand*{\Re}{\ensuremath{\mathrm{{\mathbb R}e\,}}}
	\renewcommand*{\Im}{\ensuremath{\mathrm{{\mathbb I}m\,}}}
	
	\ifnum\showllabel=1
	\def\llabel#1{\marginnote{\color{lightgray}\rm\small(#1)}[-0.0cm]\notag}
	\else
	\def\llabel#1{\notag}
	\fi
	
	\newcommand{\norm}[1]{\left\|#1\right\|}
	\newcommand{\nnorm}[1]{\lVert #1\rVert}
	\newcommand{\abs}[1]{\left|#1\right|}
	\newcommand{\NORM}[1]{|\!|\!| #1|\!|\!|}

\newtheorem{theorem}{Theorem}[section]
\newtheorem{Theorem}{Theorem}[section]
\newtheorem{corollary}[theorem]{Corollary}
\newtheorem{Corollary}[theorem]{Corollary}
\newtheorem{proposition}[Theorem]{Proposition}
\newtheorem{Proposition}[Theorem]{Proposition}
\newtheorem{Lemma}[Theorem]{Lemma}
\newtheorem{lemma}[Theorem]{Lemma}

\theoremstyle{definition}
\newtheorem{definition}{Definition}[section]
\newtheorem{Remark}[theorem]{Remark}

\def\theequation{\thesection.\arabic{equation}}
	\numberwithin{equation}{section}

	\definecolor{mygray}{rgb}{.6,.6,.6}
	\definecolor{myblue}{rgb}{9, 0, 1}
	\definecolor{colorforkeys}{rgb}{1.0,0.0,0.0}

	\newlength\mytemplen
	\newsavebox\mytempbox
	\def\weaks{\text{\,\,\,\,\,\,weakly-* in }}
	\def\weak{\text{\,\,\,\,\,\,weakly in }}
	\def\inn{\text{\,\,\,\,\,\,in }}
	\def\cof{\mathop{\rm cof\,}\nolimits}
	\def\Dn{\frac{\partial}{\partial N}}
	\def\Dnn#1{\frac{\partial #1}{\partial N}}
	\def\tdb{\tilde{b}}
	\def\tda{b}
	\def\qqq{u}
	\def\lat{\Delta_2}
	\def\biglinem{\vskip0.5truecm\par==========================\par\vskip0.5truecm}
	
	\def\inon#1{\hbox{\ \ \ }\hbox{#1}}               
	\def\onon#1{\inon{on~$#1$}}
	\def\inin#1{\inon{in~$#1$}}
	
	\def\FF{F}
	\def\andand{\text{\indeq and\indeq}}
	\def\ww{w(y)}
	\def\ll{{\color{red}\ell}}
	\def\ee{\epsilon_0}
	\def\startnewsection#1#2{ \section{#1}\label{#2}\setcounter{equation}{0}}   
	\def\nnewpage{ }
	\def\sgn{\mathop{\rm sgn\,}\nolimits}    
	\def\Tr{\mathop{\rm Tr}\nolimits}    
	\def\div{\mathop{\rm div}\nolimits}
	\def\curl{\mathop{\rm curl}\nolimits}
	\def\dist{\mathop{\rm dist}\nolimits}  
	\def\supp{\mathop{\rm supp}\nolimits}
	\def\indeq{\quad{}}           
	\def\period{.}                       
	\def\semicolon{\,;}                  
	
	\ifnum\showcolors=1
	\def\colr{\color{red}}
	\def\colrr{\color{black}}
	\def\colb{\color{black}}
	\def\coly{\color{lightgray}}
	\definecolor{colorgggg}{rgb}{0.1,0.5,0.3}
	\definecolor{colorllll}{rgb}{0.0,0.7,0.0}
	\definecolor{colorhhhh}{rgb}{0.3,0.75,0.4}
	\definecolor{colorpppp}{rgb}{0.7,0.0,0.2}
	\definecolor{coloroooo}{rgb}{0.45,0.0,0.0}
	\definecolor{colorqqqq}{rgb}{0.1,0.7,0}
	\def\colg{\color{colorgggg}}
	\def\collg{\color{colorllll}}
	\def\cole{\color{coloroooo}}
	\def\coleo{\color{colorpppp}}
	\def\colu{\color{blue}}
	\def\colc{\color{colorhhhh}}
	\def\colW{\colb}   
	\definecolor{coloraaaa}{rgb}{0.6,0.6,0.6}
	\def\colw{\color{coloraaaa}}
	\else
	\def\colr{\color{black}}
	\def\colrr{\color{black}}
	\def\colb{\color{black}}
	\def\coly{\color{black}}
	\def\colg{\color{black}}
	\def\collg{\color{black}}
	\def\cole{\color{black}}
	\def\coleo{\color{black}}
	\def\colu{\color{black}}
	\def\colc{\color{black}}
	\def\colW{\color{black}}
	\def\colw{\color{black}}
	\fi

	\def\comma{ {\rm ,\qquad{}} }            
	\def\commaone{ {\rm ,\quad{}} }          
	\def\nts#1{{\color{blue}\hbox{\bf ~#1~}}} 
	\def\ntsf#1{\footnote{\color{colorgggg}\hbox{#1}}} 
	\def\blackdot{{\color{red}{\hskip-.0truecm\rule[-1mm]{4mm}{4mm}\hskip.2truecm}}\hskip-.3truecm}
	\def\bluedot{{\color{blue}{\hskip-.0truecm\rule[-1mm]{4mm}{4mm}\hskip.2truecm}}\hskip-.3truecm}
	\def\purpledot{{\color{colorpppp}{\hskip-.0truecm\rule[-1mm]{4mm}{4mm}\hskip.2truecm}}\hskip-.3truecm}
	\def\greendot{{\color{colorgggg}{\hskip-.0truecm\rule[-1mm]{4mm}{4mm}\hskip.2truecm}}\hskip-.3truecm}
	\def\cyandot{{\color{cyan}{\hskip-.0truecm\rule[-1mm]{4mm}{4mm}\hskip.2truecm}}\hskip-.3truecm}
	\def\reddot{{\color{red}{\hskip-.0truecm\rule[-1mm]{4mm}{4mm}\hskip.2truecm}}\hskip-.3truecm}
	
	\def\tdot{{\color{green}{\hskip-.0truecm\rule[-.5mm]{3mm}{3mm}\hskip.2truecm}}\hskip-.1truecm}
	\def\gdot{\greendot}
	\def\bdot{\bluedot}
	\def\ydot{\cyandot}
	\def\rdot{\cyandot}
	
	\def\fractext#1#2{{#1}/{#2}}
	\def\ii{\hat\imath}
	\def\fei#1{\textcolor{blue}{#1}}
	\def\vlad#1{\textcolor{cyan}{#1}}
	\def\igor#1{\text{{\textcolor{colorqqqq}{#1}}}}
	\def\igorf#1{\footnote{\text{{\textcolor{colorqqqq}{#1}}}}}

	\def\AA{Y}
	\newcommand{\p}{\partial}
	\newcommand{\UE}{U^{\rm E}}
	\newcommand{\PE}{P^{\rm E}}
	\newcommand{\KP}{K_{\rm P}}
	\newcommand{\uNS}{u^{\rm NS}}
	\newcommand{\vNS}{v^{\rm NS}}
	\newcommand{\pNS}{p^{\rm NS}}
	\newcommand{\omegaNS}{\omega^{\rm NS}}
	\newcommand{\uE}{u^{\rm E}}
	\newcommand{\vE}{v^{\rm E}}
	\newcommand{\pE}{p^{\rm E}}
	\newcommand{\omegaE}{\omega^{\rm E}}
	\newcommand{\ua}{u_{\rm   a}}
	\newcommand{\va}{v_{\rm   a}}
	\newcommand{\omegaa}{\omega_{\rm   a}}
	\newcommand{\ue}{u_{\rm   e}}
	\newcommand{\ve}{v_{\rm   e}}
	\newcommand{\omegae}{\omega_{\rm e}}
	\newcommand{\omegaeic}{\omega_{{\rm e}0}}
	\newcommand{\ueic}{u_{{\rm   e}0}}
	\newcommand{\veic}{v_{{\rm   e}0}}
	\newcommand{\up}{u^{\rm P}}
	\newcommand{\vp}{v^{\rm P}}
	\newcommand{\tup}{{\tilde u}^{\rm P}}
	\newcommand{\bvp}{{\bar v}^{\rm P}}
	\newcommand{\omegap}{\omega^{\rm P}}
	\newcommand{\tomegap}{\tilde \omega^{\rm P}}
	\renewcommand{\up}{u^{\rm P}}
	\renewcommand{\vp}{v^{\rm P}}
	\renewcommand{\omegap}{\Omega^{\rm P}}
	\renewcommand{\tomegap}{\omega^{\rm P}}
	\def\hh{\text{h}}

\begin{abstract}
We address the existence and of solutions for the Euler-plate
free-boundary system modeling an interaction of a three-dimensional
inviscid fluid and an evolving plate. We prove the local existence and
uniqueness of solutions for initial fluid and structural velocities
belonging to $H^{2.5+}$ and $H^{2+}$, respectively. The results justify
earlier a~priori estimates shown by two of the authors.
\end{abstract}
	
	\maketitle
	
	\setcounter{tocdepth}{2} 
	\tableofcontents

\section{Introduction}\label{sec_int}
We consider the local-in-time well-posedness of the three-dimensional incompressible Euler-plate 
system where the evolution of the velocity is given by the three-dimensional incompressible Euler equations
\begin{align}
	\begin{split}
		& u_t 
		+ (u\cdot \nabla) u
		+ \nabla p = 0
		\\&
		\nabla \cdot u=0,
	\end{split}
	\label{EQ01}
\end{align}
in $\Omega(t) \times[0,T]$, and
the displacement of the plate is modeled by a fourth order Euler-Bernoulli equation
\begin{equation}
	w_{tt} + \Delta_2^2 w
	- \nu \Delta_{2} w_{t}
	= p
	\onon{\Gamma_1\times[0,T]}
	,
	\label{EQ02}
\end{equation}
where $p$ is the fluid pressure \colr evaluated at the interface.\colb

The model under investigation describes the interaction of a large elastic deformable structure with an inviscid incompressible fluid. We consider the case where the fluid occupies a channel with a rigid bottom and the elastic structure is at the top, and impose periodic boundary conditions in the horizontal directions. The viscous counterpart, involving the Navier-Stokes equations instead of the Euler equations, has been well-studied in the mathematical literature, beginning with \cite{DEGLT}. It was inspired by computational models of fluid-structure interaction used in the engineering and scientific communities to model a wide variety of physical systems from aeroelasticity \cite{FLT} to arterial blood flow \cite{QTV, GGCC, GGCCL}. The literature on the well-posedness of the Navier-Stokes based model is extensive and includes results on both weak solutions in both 2d and 3d \cite{CDEG} and strong solutions \cite{B,G,GH, CGH, GHL, MRR, C}. Variants of the model that include nonlinear structure equations of elasticity and Koiter shell models have also been considered in \cite{CCS, CS2, L, LR, MC1, MC2, MC3, MS, BS2, BKS} while the compressible model has been treated in \cite{BS1,BT}. Results on the model which considers the wave equation in place of the Euler-Bernoulli equations are also available in the literature \cite{L1, L2}. Linearized models and stationary models were also studied extensively as reasonable approximations,  see e.g. \cite{AB, Ch, AGW, CLW, CR, LW, W, CK}.

Unlike the viscous case,  the study of free boundary inviscid fluid-elastic structure interaction models have gained attention only recently in the mathematical community. Special interest has been given to the mathematical study of hydroelastic waves whereby potential flow (modeling  inviscid and irrotational fluid) interacts with an elastic structure modeled using the special Cosserat theory of hyperelastic shells satisfying Kirchhoff’s hypothesis (usually referred to as the Toland model) \cite{T, BT,PT1,PT2, PT3}. The Toland model involves a semilinear elliptic equation in the curvature variable of the interface and does not account for inertial effects.  Strong solutions to this model in high regularity Sobolev spaces have been established in~\cite{LA} using a vortex sheet formulation of the problem via the Birkhoff-Rott integral at the moving interface. 
Local well-posedness of the full rotational version of this model employing the Euler equations have been established in \cite{WY} recently in high regularity Sobolev spaces. The main obstacle in the mathematical study of the Euler equations based model (which is not the case for Navier-Stokes) has been the dearth of mathematical theory on variable coefficients Euler equations with non-homogeneous boundary conditions. 

For the model under consideration comprising the Euler equations to model the fluid motion and the Euler-Bernoulli beam equation describing the deformation of the elastic interface, strong solutions were constructed in a recent work~\cite{KT}, where the authors of the present paper developed a framework for studying the Euler equations with variable coefficients and non-homogeneous boundary conditions and proceeded to established the existence and uniqueness of solutions to the above Euler-plate system
assuming that the initial fluid velocity belongs to $H^3$. For the problem of weak solutions, we refer the reader to a recent work \cite{AKT} where global-in-time weak solution were established for irrotational flow in 2 space dimensions.
It is our goal in this paper to establish the local-in-time
existence and uniqueness of strong solutions under minimal Sobolev regularity of the initial
data. 
To accomplish this, we transform the Euler equations to a variable coefficients system in a fixed domain using the Arbitrary Lagrangian Eulerian (ALE) transformation which transforms the variables $(u,p)$ to new variables $(v,q)$ defined on the fixed reference initial domain $\Omega$.  The main result of the paper is Theorem 1.3, which asserts the existence of local-in-time solutions
$(v,q,w,w_t) \in L^{\infty}(H^{s+1}(\Omega) \times H^{s}(\Omega) \times H^{s+5/2}( \Gamma_{1}) \times H^{s+1/2}(\Gamma_{1})) $ for $ s>3/2$.

The proof relies on approximation of the initial data by  a smoother sequence and deriving energy estimates for the corresponding regularized sequence of solutions which are uniform in time. The main difficulty is to insure that the time of local-existence does not shrink to zero when passing to the limit. The energy estimates combine tangential and vorticity estimates with special elliptic boundary value problem bounds for the pressure. The arguments rely on an extension/continuation argument to establish a uniform time of local existence for the approximating sequence of solutions. The proof then utilizes the a~priori estimates at the critical level from \cite{KT} which  hold on a uniform time interval in addition to compactness arguments to extract convergent subsequence converging to the solution to this system with the desired regularity. Finally, the uniqueness of solutions is established by considering energy estimates for the difference of two solutions in a lower topology. This is necessitated by the fact that estimates on the difference of the vorticity variables corresponding to two different solutions satisfy an equation which is no longer of purely transport type whereby higher order terms appear and for which estimates can only be closed by considering one derivative below the maximum level. This in turn requires estimates for the pressure as a solution to an elliptic boundary value problem with Robin type data at a low regularity level. These elliptic estimates are adapted from classical works \cite{LM} to the situation at hand with variable coefficients.

The paper is organized as follows. In the next section, we perform the ALE change of variables, while Section~\ref{sec12} contains the statement of the main theorems. The detailed proofs of the two main theorems concerning the existence and uniqueness are provided in
Sections~\ref{sec_ex} and~\ref{sec.un}, respectively.

\section{The Euler-plate system and the main results}
\label{sec2}
\subsection{The ALE coordinates}

For any $t>0$, consider the solution $\psi$
to the elliptic problem
\begin{align}
	\begin{split}
		&\Delta \psi = 0
		\inon{on $\Omega$}
		\\&
		\psi(x_1,x_2,1,t)=1+w(x_1,x_2,t)
		\inon{on $\Gamma_1$}
		\\&
		\psi(x_1,x_2,0,t)=0
		\inon{on $\Gamma_0$}
		,
	\end{split}
	\label{EQ03}
\end{align}
which is the harmonic extension of
$1+w$ to the domain $\Omega=\Omega(0)$ for $t\in[0,T]$.
We use $\psi$ to define the ALE change of coordinates map 
$\eta\colon \Omega\times[0,T]\to \Omega(t) $ as
\begin{equation}
	\eta(x_1,x_2, x_3,t)=(x_1,x_2,\psi(x_1,x_2,x_3,t))
	\comma (x_1,x_2,x_3)\in \Omega
	,
	\llabel{EQ04}
\end{equation}
with the gradient
\begin{equation}
	\nabla \eta
	=
	\begin{pmatrix}
		1 &  0 & 0 \\
		0 &  1 & 0 \\
		\partial_{1}\psi & \partial_{2} \psi & \partial_{3}\psi
	\end{pmatrix}
	.
	\llabel{EQ05}
\end{equation}
Denote by $a$ the inverse of this derivative, i.e.,
\begin{equation}
	a = \frac{1}{J} b
	=
	\begin{pmatrix}
		1 &  0 & 0 \\
		0 & 1 & 0 \\
		-\partial_{1}\psi/\partial_{3}\psi & -\partial_{2}\psi/\partial_{3}\psi & 1/\partial_{3}\psi
	\end{pmatrix}
	,
	\label{EQ06}
\end{equation}
where $J$ is the Jacobian and $b$ is the cofactor matrix, i.e.,
\begin{equation}
	J=\partial_{3}\psi,
	\label{EQ07}
\end{equation}
and 
\begin{equation}
	b
	=
	\begin{pmatrix}
		\partial_{3}\psi &  0 & 0 \\
		0 & \partial_{3} \psi & 0 \\
		-\partial_{1}\psi & -\partial_{2}\psi & 1
	\end{pmatrix}
	.
	\label{EQ08}
\end{equation}
We note in passing that 
\begin{align}
	\Vert \eta\Vert_{H^{s}}
	 \lec 
	  \Vert \psi\Vert_{H^{s}}
	   \lec
	    \Vert w\Vert_{H^{s-0.5}(\Gamma_{1})}
   \andand
	\Vert J\Vert_{H^{s}}
	\lec
	\Vert \psi\Vert_{H^{s+1}}
	\lec
	\Vert w\Vert_{H^{s+0.5}(\Gamma_{1})}
	    ,\llabel{EQ09}
\end{align}
and that
\begin{align}
	\Vert b\Vert_{H^{s}}
	 \lec
	  \Vert \psi\Vert_{H^{s+1}}
	   \lec
	    \Vert w\Vert_{H^{s+0.5}(\Gamma_{1})}
	    ,\llabel{EQ10}
\end{align}
for sufficiently regular solutions and $s\ge 0$.
Assuming that $J$ remains bounded below 
by a positive constant, we also have
\begin{align}
	\Vert a\Vert_{H^{s}}
	\lec
	\Vert w\Vert_{H^{2+\delta}(\Gamma_{1})}
	 \Vert w\Vert_{H^{s+0.5}(\Gamma_{1})}
	,\llabel{EQ11}
\end{align}
for $s\ge 0$. Upon differentiating in time,
we may also estimate $a_t$, $b_t$, and $\psi_t$
in terms of $w$ and~$w_t$.
Next, we write the Piola identity
\begin{equation}
	\partial_{i}b_{ij}=0
	\comma j=1,2,3
	,
	\label{EQ12}
\end{equation}
which follows directly from~\eqref{EQ08}.
As in \eqref{EQ12}, we sum over the repeated indices throughout the paper
unless otherwise indicated.
Now, we denote the ALE velocity and pressure by
\begin{align}
	\begin{split}
		&   v(x,t) = u(\eta(x,t),t)
		\\&
		q(x,t) = p(\eta(x,t),t)
,
	\end{split}
	\llabel{EQ13}
\end{align}
and with the ALE coordinates the Euler-plate system 
\eqref{EQ01} is given by
\begin{align}
	\begin{split}
		&
		\partial_{t} v_i
		+ v_k a_{jk} \partial_{j}v_i
		- \frac{1}{\partial_{3}\psi}\psi_t \partial_{3} v_i
		+ a_{ki}\partial_{k}q
		=0
		,
		\\&
		a_{ki} \partial_{k}v_i=0
	\end{split}
	\label{EQ14}
\end{align}
in $\Omega\times[0,T]$.
The initial conditions read
\begin{equation}
	(v,w,w_t)|_{t=0} = (v_0,0,w_1)
	.
	\label{EQ15}
\end{equation}
The Eulerian slip condition on the bottom boundary becomes
\begin{equation}
	v_3=0
	\inon{on $\Gamma_0$}
	,
	\label{EQ16}
\end{equation}
while the velocity matching condition on the top boundary
is given by
\begin{equation}
	b_{3i}v_i = w_t
	\inon{on $\Gamma_1$}
	,
	\label{EQ17}
\end{equation}
where we used \eqref{EQ08} and~\eqref{EQ03}$_2$.
Next, the plate equation \eqref{EQ02} translates to
\begin{align}
	w_{tt} 
	+\Delta_\hh^2 w
	- \nu \Delta_\hh w_{t}
	= q
	,
	\label{EQ18}
\end{align}
where the pressure satisfies
\begin{equation}
	\int_{\Gamma_1} q = 0
	,
	\label{EQ21}
\end{equation}
for all $t\in[0,T]$.

Throughout the paper, we denote the spaces $L^p((0,T);X(\Omega))$ and
$C([0,T];X(\Omega))$ by  $L^p_TX$ and $C_TX$, respectively. 
In addition, we denote by $C\geq1$ a sufficiently large generic constant that is
independent of the data and may change from line to line. 

\subsection{The main result}
\label{sec12}
Fix $\delta\in (0,1/2]$, and consider
$(v_0,w_1) \in H^{2.5+\delta }\times H^{1+\delta}(\Gamma_1)$. 
We are now ready to state our main result.

\cole
\begin{Theorem}[Local existence]
\label{T01}
Let $\nu \in [0,1]$ and 
$(v_0,w_1) \in H^{2.5+\delta }\times H^{1+\delta}(\Gamma_1)$. 
Moreover, assume the compatibility conditions
\begin{equation} 
	v_{0}\cdot N |_{\Gamma_{1}}
	=w_1
     \andand
	v_{0}\cdot N |_{\Gamma_{0}}=0,
	\label{EQ19}
\end{equation} 
with
\begin{equation} 
	\div v_{0}=0 \inon{in~$\Omega$}
      \andand
	\int_{\Gamma_1} w_1 = 0
	.
	\label{EQ20}
\end{equation}
Then, there exists
a unique solution $(v,q,w)$
on $[0,T]$ to the Euler-plate system \eqref{EQ14}--\eqref{EQ21} such that
\begin{align}
	\begin{split}
		&v \in L^{\infty}([0,T];H^{2.5+\delta}(\Omega))
		\cap C([0,T];H^{2.5+\delta^-}(\Omega))
		,
		\\&
		v_{t} \in L^{\infty}([0,T];H^{0.5+\delta}(\Omega))
		,
		\\&
		q \in L^{\infty}([0,T];H^{1.5+\delta}(\Omega))
		, 
		\\&
		w \in L^{\infty}([0,T];H^{4+\delta}(\Gamma_{1}))
		\cap C([0,T];H^{4+\delta^-}(\Gamma_1))
		,
		\\&
		w_{t} \in L^{\infty}([0,T];H^{2+\delta}(\Gamma_{1}))
		\cap C([0,T];H^{2+\delta^-}(\Gamma_1)) 
		,
	\end{split}
	\label{EQ22}
    \end{align}
where
$T=T(\Vert v_0\Vert_{H^{2.5+\delta}},
\Vert w_1\Vert_{H^{1+\delta}(\Gamma_1)})>0$.
\end{Theorem}
\colb

Above, $\delta^{-}>0$ denotes any positive constant
less than~$\delta$.
As mentioned, the a~priori bounds for this existence result are established in~\cite{KT}.
Before recalling them, we remark that the functions $a$, $b$, and $J$ given by \eqref{EQ06}, \eqref{EQ07},
and \eqref{EQ08}, respectively, are close to the identity for a sufficiently small
amount of time, and this is key to our proof. 
Now, we denote by $C_\text{abs}>0$
an absolute constant
much larger than the largest $C$ appearing in the text below,
and consider $0<\epsilon_0 < 1/2C_\text{abs}$. This
is needed to close our estimates; see~Section~\ref{sec.con}.
Next, we fix
\begin{align}
	0<\epsilon<\frac{\epsilon_0}{2C_\text{abs}}
	.\label{EQ23}
\end{align}
This choice
suffices for the
continuation argument when we construct solutions 
as well as the absorption arguments 
we employ in the pressure estimates
\eqref{EQ55} and \eqref{EQ58}--\eqref{EQ61}, div-curl estimates
\eqref{EQ131}--\eqref{EQ134}, and in concluding
the proof of Proposition~\ref{P01}, i.e.,~\eqref{EQ136}.  
\colb
Now, we cite~\cite[Theorem 2.1]{KT} which gives us 
needed a~priori estimates, including the smallness of
important quantities; see also~\cite[Lemma 3.1]{KT}.

\cole
\begin{Theorem}[A~priori~estimates~for~existence]
	\label{T02}
	Let $\nu\in [0,1]$, and 
	assume that
  \begin{equation*}
	(v,w) \in (L^{\infty}_TH^4, L^{\infty}_TH^{5.5})   
  \end{equation*}
	is a solution
	of the Euler-plate system \eqref{EQ14}--\eqref{EQ21}
	on an interval $[0,T]$ with
	\begin{align}
		\begin{split}
			\Vert v_0\Vert_{H^{2.5+\delta}},
			\Vert w_1\Vert_{H^{2+\delta}(\Gamma_1)}
			\leq M   
			,
		\end{split}
		\label{EQ24}
	\end{align}
	where $M\geq1$.
	Then, there exists $T_0>0$ such that
	\begin{align} 
		\begin{split} 
			&\Vert v\Vert_{H^{2.5+\delta}}, 
			\Vert w\Vert_{H^{4+\delta}(\Gamma_1)}, 
			\Vert w_t\Vert_{H^{2+\delta}(\Gamma_1)}, 
			\Vert \psi\Vert_{H^{4.5+\delta}}, 
			\Vert \psi_t\Vert_{H^{2.5+\delta}}, \Vert a\Vert_{H^{3.5+\delta}}
			\leq C_0 M
			,
		\end{split}
		\label{EQ25}
	\end{align}
	with
	\begin{equation}
		\nu^{1/2} \Vert w_t\Vert_{L^2H^{3+\delta}(\Gamma_1\times[0,\min\{T_0,T\}])} \leq C_0 M   
		\label{EQ26}
	\end{equation}
	and
	\begin{align}
		\begin{split}
			&   \Vert v_t\Vert_{H^{0.5+\delta}},
			\Vert w_{tt}\Vert_{H^{\delta}(\Gamma_1)},
			\Vert q\Vert_{H^{1.5+\delta}}
			\leq C_0 M
			,
		\end{split}
		\label{EQ27}
	\end{align}
	where $ t\in[0,\min\{T_0,T\}]$, and $C_0\geq2$ is a constant.
Moreover, 
	  \begin{align}
	  	\Vert a-I\Vert_{H^{1.5+\delta}},
	  	\Vert \tda-I\Vert_{H^{1.5+\delta}},
	  	\Vert J-I\Vert_{H^{1.5+\delta}},
	  	\Vert J-I\Vert_{L^{\infty}}
	  	< \epsilon
	  	\comma t\in [0,\min\{T_0,T\}],
	  	\label{EQ28}
	  \end{align}
where $\epsilon>0$ is given by \eqref{EQ23},
and $T_0$ depends on $M$ and~$\epsilon$.   
\end{Theorem}
\colb

In particular, $C_0$ and $T_0$ do not depend on~$\nu$.

We remark that in~\cite{KT} solutions are assumed to be $C^\infty$ smooth in establishing the above theorem. 
However, the proof holds for the regularity stated in
Theorem~\ref{T02}.
We now
fix $C_0\geq1$ for the rest of this paper and omit stating any dependence on it.

We aim to utilize the a~priori estimates and obtain a sequence of uniformly bounded
solutions. To achieve this, we first regularize the initial data in \eqref{EQ34}--\eqref{EQ38},
then we invoke the existence result~\cite[Theorem 2.3]{KT}, stated
next,
and pass to the limit.

\cole
\begin{Theorem}
	\label{T03}
	(Local existence with regular data)
	Let $\nu \in [0,1]$. 
	Assume that initial data
	\begin{equation}
		(v_{0}, w_{1})\in H^{4} \times   H^{3.5}(\Gamma_{1})
		\comma \Vert v_0\Vert_{H^{4}}+\Vert w_1\Vert_{H^{3.5}(\Gamma_1)}\le \bar{M}
		,\llabel{EQ29}
	\end{equation} 
	for some $\bar{M}>0$,
	satisfy the
	compatibility conditions \eqref{EQ19} and~\eqref{EQ20}.
	Then there exists a unique local-in-time solution $(v,q,w)$ to the 
	Euler-plate system 
	\eqref{EQ14}--\eqref{EQ21} with the initial data $(v_0,w_1)$
	such that
	\begin{align}
		\begin{split}
			&v \in L^{\infty}([0,T];H^{4}(\Omega))
			\cap C([0,T];H^{4^-}(\Omega))
			,
			\\&
			v_{t} \in L^{\infty}([0,T];H^{2}(\Omega))
			,
			\\&
			q \in L^{\infty}([0,T];H^{3}(\Omega))
			, 
			\\&
			w \in L^{\infty}([0,T];H^{5.5}(\Gamma_{1}))
			\cap C([0,T];H^{5.5^-}(\Gamma_1))
			,
			\\&
			w_{t} \in L^{\infty}([0,T];H^{3.5}(\Gamma_{1}))
			\cap C([0,T];H^{3.5^-}(\Gamma_1)) 
			,
		\end{split}
		\label{EQ30}
	\end{align}
	and
	\begin{align}
		\Vert v\Vert_{L^\infty_T H^{4}} + \Vert w\Vert_{L^\infty_T H^{5.5}(\Gamma_1)}
		+\Vert w_t\Vert_{L^\infty_T H^{3.5}(\Gamma_1)} \le K \bar{M}
		,\llabel{EQ31}
	\end{align}
	for some constant $K>1$ and time $T>0$ depending on $\bar{M}$ and~$K$.
	In particular, both $K$ and $T$ are independent of~$\nu$.
\end{Theorem}
\colb

When we employ this theorem for each member
of the approximating sequence,
the time of existence may shrink to zero
in the limit. This is
because higher norms of the 
initial data may grow indefinitely.
Therefore, we may need to extend the lifespan of solutions
to obtain a uniform time of existence
for this sequence.
We achieve this by 
establishing a linear ODE inequality
on the higher norms of the solutions
as long as they exist and \eqref{EQ28} holds.
More precisely, we prove the following
statement.

\cole
\begin{Proposition}
\label{P01}
Let $\nu \in [0,1]$ and suppose that
	$(v,w,q)$
	is a solution of the Euler-plate
	system~\eqref{EQ14}--\eqref{EQ21} on~$[0,T]$,
	as described in~\eqref{EQ30}.
	In addition, assume that
	\eqref{EQ25}--\eqref{EQ28}
	hold on $[0,\min\{T_0,T\}]$ 
	for $M>0$,
	where $\epsilon>0$
	is given by \eqref{EQ23}
	and $T_0>0$ is 
	given by Theorem~\ref{T02}.
	Then, for 
	\begin{align}
		Q(t)
		=
		(\Vert v\Vert_{H^{4}}
		+\Vert w\Vert_{H^{5.5}(\Gamma_{1})}
		+\Vert w_t\Vert_{H^{3.5}(\Gamma_{1})}
		+1)^2
		,
		\label{EQ32}   
	\end{align}
	we have
	\begin{align}
		Q(t)+\nu \int_{0}^{t} \Vert w_t\Vert_{H^{4.5}(\Gamma_{1})}^2
		 \le
		 CQ(0)+\int_{0}^{t}C_MQ(s)\,ds
		 \comma
		 t \in [0,\min\{T_0,T\}]
		 ,\label{EQ33}
	\end{align}
	where the constants $C,C_M>0$ are independent of
	$Q(0)$ and $\nu$, and $C>0$ is independent of~$M$.
	\end{Proposition}
\colb

This proposition guarantees the at-most-exponential growth of approximating solutions
on the time interval $[0,\min\{T_0,T\}]$.
Combining this with the choice of $\epsilon$
given in \eqref{EQ23}, we 
obtain a uniform time of existence.
Then, we pass to the limit to obtain a solution to the
Euler-plate system with the regularity specified in Theorem~\ref{T01}.

In Section~\ref{sec_ex},
we first prove the existence part of Theorem~\ref{T01}
assuming that Proposition~\ref{P01} holds.
Next, we establish Proposition~\ref{P01}.
This is done in three steps,
presented in Section~\ref{sec.pre}, \ref{sec.tan},
and~\ref{sec.vor} for the pressure, tangential and
vorticity estimates, respectively.
Then, in Section~\ref{sec.con}, we collect them to conclude 
the proof of Proposition~\ref{P01}. Finally,
in Section~\ref{sec.un}, we establish the uniqueness of solutions.

\section{Proof of Theorem~\ref{T01}: Existence}\label{sec_ex}

Here, we present the proof of the existence part of Theorem~\ref{T01}.
We use $(\bar{v}_0,\bar{w}_1)$ instead of $(v_0,w_1)$
to denote the initial data.

\begin{proof}[Proof of Theorem~\ref{T01}: Existence]
Fix $\nu \in [0,1]$, and 
consider $(\bar{v}_0,\bar{w}_1) \in H^{2.5+\delta} \times H^{2+\delta}$
satisfying the compatibility conditions \eqref{EQ19} and~\eqref{EQ20}.
Moreover, assume \eqref{EQ24} holds.
First, we regularize the initial data.
Let $r \in (0,1]$
denote the approximation parameter. Since $\bar{w}_1$ is periodic, 
by truncating the Fourier modes we define
	\begin{align}
		w_1^{(r)}(x)
		= \sum_{|k|\le 1/r}
		\hat{\bar{w}}_1(k)
		e^{-2\pi i k\cdot x}
		,\label{EQ34}
	\end{align} 
so that $w_1^{(r)} \in C^\infty(\Gamma_{1})$ and 
$w_1^{(r)} \to \bar{w}_1$ in $H^{4+\delta}$.
Since $\bar{w}_1$ is average-free, we have
$\hat{\bar{w}}_1(0) = 0$ so that
	\begin{align}
		\int_{\Gamma_1} w_1^{(r)} = 0
		.
		\label{EQ35}
	\end{align}
Next, for $\bar{v}_0$,
we use the approximation scheme in~\cite{AKOT} (see also~\cite{CS1})
adapted to our setting.
First, we consider $\phi_r\in C_0^\infty (\mathbb{R}^3)$,
a family of standard mollifiers with 
$\supp \phi_r \subseteq B(0, r/2 )$, and define
	\begin{align}
		\begin{split}
			\tilde v_0^{(r)}=\begin{pmatrix} (v_0)_1 (x_1,x_2,(x_3+r)/(1+2r)) \\
				(v_0)_2 (x_1,x_2,(x_3+r)/(1+2r)) \\
				(1+2r)^{-1}(v_0)_3 (x_1,x_2,(x_3+r)/(1+2r)) 
			\end{pmatrix}
\andand
			g^{(r)}(x) =
			\phi_{r} * \tilde v_0^{(r)}
		\end{split}
		\llabel{EQ36}
	\end{align}
with
  \begin{equation}
   v_0^{(r)} = g^{(r)} - \nabla h^{(r)}
   ,
   \llabel{EQ194}
  \end{equation}
where $h^{(r)}$ uniquely (up to a constant) solves the PDE 
	\begin{align}
		\begin{split}
			\Delta h^{(r)}  &= \div g^{(r)}
			\hspace{0.53cm} \inon{in $\Omega$},
			\\
			\partial_3 h^{(r)} &=  g_3^{(r)}-w_1^{(r)}
			\inon{on $\Gamma_1$},
			\\
			\partial_3 h^{(r)} &= g_3^{(r)}
			\hspace{1.05cm}      \inon{on $\Gamma_0$}     
		\end{split}
		\label{EQ37}
	\end{align}
with the periodic conditions in horizontal variables.
Note that \eqref{EQ35} guarantees that the Neumann problem is
solvable.
Thus we have $\tilde v_0 \in C^\infty(\mathbb{T}^2\times(-r,1+r))$,
and consequently,
$g^{(r)}$ is smooth in a neighborhood of~$\Omega$.
Letting $r\to 0$ and using \eqref{EQ19} with
$w_1^{(r)} \to \bar{w}_1$,
we conclude that $h^{(r)}$ converges to
a unique, up to a constant, solution of
\begin{align}
		\begin{split}
			\Delta h
			&= 0
			\inon{in $\Omega$}
			\\
			\partial_3 h &= 0
			\inon{on $\Gamma_0 \cup \Gamma_{1}$}
			,
		\end{split}
		\label{EQ38}
	\end{align}
	implying that $\nabla h^{(r)}\to 0$, as $r\to0$. 
	By the standard elliptic theory, $v_0^{(r)} \in C^\infty(\Omega)$ and $v_0^{(r)} \rightarrow \bar{v}_0$ in $H^{2.5+\delta}$ as $r \rightarrow 0$. 
	
Now, we are in a position to invoke Theorem~\ref{T03}.
Recalling \eqref{EQ23} and that $\Cabs>0$
is an absolute constant,
we first fix $K>\Cabs$ sufficiently large. Then,
by Theorem~\ref{T03},
we obtain
a unique solution $(v^{(r)},q^{(r)},w^{(r)})$, for $r\in (0,1)$,
to the Euler-plate system satisfying \eqref{EQ30}
on $[0,T^{(r)}]$.
	
	Now, we need a time interval, independent of $r$,
	on which our approximate solutions exist.
	To find one, we first recall that $T_0>0$, given by Theorem~\ref{T02}, only 
	depends on $\epsilon$ given in \eqref{EQ23}
	and the lower norms of the initial data, i.e., $M>0$, a fixed quantity.
	Therefore, $T_0$ does not depend on $r>0$. 
	Next, we employ Proposition~\ref{P01} and  
	conclude that $Q^{(r)}$ given by 
	\begin{align}
		Q^{(r)}(t)
		=
		(\Vert v^{(r)}\Vert_{H^{4}}
		+\Vert w^{(r)}\Vert_{H^{5.5}(\Gamma_{1})}
		+\Vert w_t^{(r)}\Vert_{H^{3.5}(\Gamma_{1})}
		+1)^2
		,
		\llabel{EQ39}   
	\end{align}
	satisfies the growth bound
	\begin{align}
		Q^{(r)}(t)
		 \le
		  CQ^{(r)}(0)e^{C_MT}
		  ,\label{EQ40}
	\end{align}
	for any $t\in [0,\min\{T_0,T^{(r)}\}]$.
	On the other hand, Theorem~\ref{T03}
	guarantees the bound
	\begin{align}
		Q^{(r)}(t)
		\le
		 K Q^{(r)}(0)
		 ,\label{EQ41}
 	\end{align}
 	so that we define the break-even time $\bar{T}_0>0$ as
 	\begin{align}
 		Ce^{C_M\bar{T}_0} = K
 		.\label{EQ42}
 	\end{align}
 	We note that this is possible since $K>\Cabs$
 	and also that $\bar{T}_0$ is independent of $r$.
 	Finally, we let $T^* = \min\{T_0,\bar{T}_0\}>0$
 	and claim that the lifespan of approximate
 	solutions may be extended to $[0,T^*]$.
 	
 	Before presenting our continuation argument
 	for solutions that have
 	$T^{(r)} < T^*$, we first briefly summarize the
 	construction scheme leading to Theorem~\ref{T03}.
 	In~\cite{KT}, the authors 
 	first solve \eqref{EQ14}--\eqref{EQ17}, i.e., the
 	Euler equations with variable coefficients
 	where $(w,w_t,a,b,\psi,J)$ are given.
 	At this stage, they assume that the initial
 	displacement $w_0$ is zero, and this is imposed 
 	for two reasons. First, when they make the ALE
 	change of variable, the zero initial displacement
 	flattens the top boundary. Second, this assumption
 	guarantees that $a$, $b$, and $J$ stay close 
 	to the identity for a sufficiently short time,
 	and this enables the absorption tricks akin to 
 	\eqref{EQ55}, \eqref{EQ58}--\eqref{EQ61}, 
 	\eqref{EQ131}--\eqref{EQ134}, and~\eqref{EQ136}.  
 	Next, the authors solve the plate equation \eqref{EQ18}
 	given the pressure term without imposing
 	that $w_0 = 0$. Finally, upon utilizing a fixed point 
 	argument, they conclude the existence part of Theorem~\ref{T03};
 	see~\cite[Section~6]{KT}.
 	
 	Now, we fix $r>0$ and assume that $T= T^{(r)} < T^*$,
 	Recalling that our solution
 	has the regularity given in \eqref{EQ30},
 	we may find a $\tau_0\in (0,T)$
 	so that
 	\begin{align}
 		(v(\tau_0),w(\tau_0),w_t(\tau_0)
 		) \in
 		 H^4 \times H^{5.5}(\Gamma_1) \times H^{3.5}(\Gamma_1)
 		 .\label{EQ43}  
 	\end{align} 	
Moreover, we can verify that this triple
satisfies the compatibility conditions \eqref{EQ19} and \eqref{EQ20}
 as in~\cite[Section~4]{KT}. Therefore,
 	 we may repeat the construction steps summarized
 	 above to get a unique solution of the Euler-plate system 
 	 with the initial datum given in~\eqref{EQ43}.
 	 We note that $w(\tau_0)$ need not be equal to zero,
 	 and this does not interfere with the strategy outlined above.  
Indeed, since we are already in the ALE variables and the top boundary
     is flat, we do not perform an additional change of coordinates
     when we solve the Euler-plate system with the initial datum~\eqref{EQ43}.
     Moreover, \eqref{EQ28} at $t=\tau_0$ and
     \eqref{EQ23} suffices for absorption-type arguments
     since $\epsilon>0$ does not depend on the data.
     Therefore,
     we obtain the existence of a unique solution
     $(\tilde{u},\tilde{w})$ on $[\tau_0,\tau_0+\tau]$,
     for some $\tau>0$.
     Utilizing the uniqueness again, we conclude
     that the solution $(u,w)=(\tilde{u},\tilde{w})$ exists on
     $[0,\max\{T,\tau_0+\tau\}]$.
     If we have $\tau_0+\tau \ge T^*$,
     we do not continue the solutions further.
     On the other hand, if $\tau_0+\tau < T^*$,
     then we either have
     \begin{align}
     	\tau_0< \tau_0+\tau \le T 
     	,\label{EQ44}
     \end{align}
     or
     \begin{align}
     	\tau_0 < T < \tau_0+\tau
     	.\label{EQ45}
     \end{align}
     When \eqref{EQ44} takes place, we can find 
     $\tau_1 \in [T-\tau/2,T]$
     so that \eqref{EQ43} is satisfied
     upon writing $\tau_1$ instead of $\tau_0$.
     Since $\tau_1<T^*$, it follows that $\eqref{EQ28}$ is valid for $t=\tau_1$,
     and since $\tau_1<T$, the initial datum $(u(\tau_1),w(\tau_1),w_t(\tau_1))$
     satisfies the bound in~\eqref{EQ41}. 
     Therefore, we may construct a unique solution to the Euler-plate system with
     this initial datum and employ uniqueness again to obtain
     the existence of $(u,w)$ on the time interval $[0,T+\tau/2]$.
     Now, we assume that \eqref{EQ45} is true and find 
     $\tau_1 \in [\tau_0+\tau/2,\tau_0+\tau]$
     such that \eqref{EQ43} holds upon replacing $\tau_0$
     with $\tau_1$.
     Note that this initial datum satisfies the bound~\eqref{EQ41}.
     To see this, we employ the exponential growth \eqref{EQ40}
     which holds as long as the solution exists and $t<T_0$.
     Recalling \eqref{EQ42} and that $T^*=\min\{T_0,\bar{T_0}\}$,
    we conclude that
    \begin{align}
    	Q(\tau_1)
    	\le
    	CQ(0)e^{C_M\tau_1}
    	 \le
    	  CQ(0)e^{C_M\bar{T}_0}
    	  =Q(0)K
    	,\llabel{EQ46}
    \end{align}
      and this guarantees the existence of the solution 
      on the time interval $[0,\tau_0+3\tau/2]$.
      If $\tau_0+3\tau/2<T^*$,
      we repeat this procedure until  
      the lifespan of the solution is extended to $[0,T^*]$.
      Therefore, we may assume that 
      for all $r>0$, the unique solution $(u^{(r)},w^{(r)})$
      exists on $[0,T^*]$ satisfying the bounds given in 
      Theorem~\ref{T02}.
      
    Now, upon passing to a subsequence $r \to 0^+$,
	there exists $(v,w,q,\eta,a)$ such that
	\begin{align}
		\begin{split}
			&v^{(r)} \to v \weaks L^{\infty}([0,T^*];H^{2.5+\delta}),
			\\&
			v_{t}^{(r)}\to v_t \weaks  L^{\infty}([0,T^*];H^{0.5+\delta}),
			\\&
			q^{(r)}\to q  \weaks \in  L^{\infty}([0,T^*];H^{1.5+\delta}),
			\\&
			w^{(r)}\to w \weaks \in  L^{\infty}([0,T^*];H^{4+\delta}(\Gamma_{1})),
			\\&
			w_{t}^{(r)}\to w_{t} \weaks L^{\infty}([0,T^*];H^{2+\delta}(\Gamma_{1})),
			\\&
			w_{tt}^{(r)} \to w_{tt} \weaks  L^{\infty}([0,T^*];H^{\delta}(\Gamma_{1})),
			\\&
			\eta^{(r)} \to \eta \weaks  L^{\infty}([0,T^*];H^{4.5+\delta}),
			\\&
			a^{(r)} \to a \weaks L^{\infty}([0,T^*];H^{3.5+\delta}),
			\\&
			a_{t}^{(r)} \to a_{t} \weaks  L^{\infty}([0,T^*];H^{1.5+\delta})
			.
		\end{split}
		\llabel{EQ47}
	\end{align}
Furthermore, employing
	the Aubin-Lions lemma yields 
	\begin{align}
		\begin{split}
			&v^{(r)}\to v \inn  C([0,T^*];H^{s})
			\\&
			a^{(r)}\to a \inn  C([0,T^*];H^{s+1})
			,
		\end{split}
		\llabel{EQ48}
	\end{align}
	for any $s < 2.5+\delta$.
	Finally, we can pass to the limit
	in the equations as in~\cite{KT},
	concluding the proof of the existence part
	of Theorem~\ref{T01}.	
\end{proof}

In this the rest of this section, 
we find
estimates required to prove Proposition~\ref{P01}.

\subsection{Pressure estimates}\label{sec.pre}
Here, we present the necessary estimates for the pressure.
From \eqref{EQ14}, it 
follows that $q$ satisfies
\begin{align}
	\begin{split}
		\partial_{j}(\tda_{ji} a_{ki}\partial_{k}q)   
		&
		=
		\partial_{j}(\partial_{t}\tda_{ji} v_i)
		-
		b_{ji} \partial_{j}(v_m a_{km}) \partial_{k} v_i
		+ b_{ji}
		\partial_{j}\left(\frac{\psi_t}{J}\right)\partial_{3}v_i
 	\\&\indeq\indeq
		+
		v_m a_{km} \partial_{k}b_{ji} \partial_{j}v_i
		- 
		\frac{\psi_t}{J}\partial_{3}b_{ji}\partial_{j}v_i
		=\fomega
		\inon{in $\Omega$}
		,\label{EQ49}
	\end{split}
\end{align}
with the boundary conditions
\begin{align}
	\begin{split}
		\tda_{3i}a_{ki}\partial_{k}q
		= 0
		\inon{on $\Gamma_0$}
		,
	\end{split}
	\llabel{EQ50}
\end{align}
and
\begin{align}
	\begin{split}
		b_{3i}a_{ki}\partial_{k}q
		+ q
		&=\Delta_2^2 w 
		-  \nu   \Delta_2   w_{t}
		+ \partial_{t}\tda_{3i}v_i
		-
		J^{-1}
		\sum_{j=1}^{2}  v_k b_{jk} \partial_{j}w_t
		+ J^{-1}w_t \partial_{3}b_{3i} v_i
		-   J^{-1}    \partial_{j}      b_{3i} v_k b_{jk} v_i
		\\&=\fgamma
		\inon{on $\Gamma_1$}
		.
	\end{split}
	\label{EQ51}
\end{align}
For the derivation of \eqref{EQ49}--\eqref{EQ51}, see~\cite{KT}.
The equations \eqref{EQ49}--\eqref{EQ51} may be rewritten as
\begin{align}
	\begin{split}
		-\Delta q
		&= \partial_j((b_{ji}a_{ki}-\delta_{jk})\partial_kq)-\fomega
		,\inon{in $\Omega$}
		\\
		\partial_3 q &= (b_{3i}a_{ki}-\delta_{3k})\partial_k q
		,\hspace{1.4cm} \inon{on $\Gamma_{0}$}
		\\
		\partial_3 q+ q &= (b_{3i}a_{ki}-\delta_{3k})\partial_k q-\fgamma
		,\hspace{0.5cm} \inon{on $\Gamma_{1}$}
		.\label{EQ52}
	\end{split}
\end{align}
The next lemma provides the pressure estimates, based on~\eqref{EQ52}

\cole
\begin{Lemma}
	\label{L.Pre}
Let $\nu \in [0,1]$ and
suppose that
	$(v,w,q)$
	is a solution of the
	Euler-plate system \eqref{EQ14}--\eqref{EQ21}
	satisfying \eqref{EQ30}
	on $[0,T]$.	
	In addition, assume that
	\eqref{EQ25}--\eqref{EQ28}
	hold on $[0,\min\{T_0,T\}]$ 
	for some $M>0$,
	where $\epsilon>0$
	is given by \eqref{EQ23}
	and $T_0>0$ is 
	given by Theorem~\ref{T02}.
	Then, with
	\begin{align}
		Q(t)
		=
		(\Vert v\Vert_{H^{4}}
		+\Vert w\Vert_{H^{5.5}(\Gamma_{1})}
		+\Vert w_t\Vert_{H^{3.5}(\Gamma_{1})}
		+1)^2
		,
   \llabel{EQ53}
	\end{align}
	we have
	\begin{align}
		\Vert q(t)\Vert_{H^{3}}^2
		\le C_M Q(t)
		\comma
		t\in [0,\min\{T_0,T\}],
		\label{EQ54}
	\end{align} 
	where $C_M>0$ depends only on $M$
	and in particular, neither on $\nu$ nor on~$Q(0)$.
	\end{Lemma}
\colb

\begin{proof}[Proof of Lemma~\ref{L.Pre}]
We apply the $H^3$ elliptic estimates (see~\cite{LM}) to the problem~\eqref{EQ52}
obtaining
	\begin{align}
  	\begin{split}
		\Vert q\Vert_{H^{3}}
		&\le
		C\bigl(\Vert \partial_j((b_{ji}a_{ki}-\delta_{jk})\partial_kq)\Vert_{H^{1}}
		+\Vert (b_{3i}a_{ki}-\delta_{3k})\partial_k q\Vert_{H^{1.5}(\Gamma_{0} \cup \Gamma_{1})}
     \\&\indeq\indeq\indeq\indeq\indeq
		+\Vert \fomega\Vert_{H^{1}}
		+\Vert \fgamma\Vert_{H^{1.5}(\Gamma_{1})}
		\bigr).
  \end{split}
		\label{EQ55}
	\end{align}
For the first term on the right-hand side of \eqref{EQ55},
	we write
	\begin{align}
		\Vert (ba-I)\nabla q\Vert_{H^{2}}
		 \le
		  \Vert (b-I)a\nabla q\Vert_{H^{2}}+\Vert (a-I)\nabla q\Vert_{H^{2}}
		  .\label{EQ56}
		  \end{align}
	Now, we recall the Sobolev product inequality
	\begin{align}
		\Vert fg\Vert_{W^{s,p}}
		\le
		C(\Vert f\Vert_{W^{s,p_1}}
		\Vert g\Vert_{L^{p_2}}
		+
		\Vert f\Vert_{L^{q_1}}
		\Vert g\Vert_{W^{s,q_2}}
		),
		\llabel{EQ57}
	\end{align}
	where $s\ge 0$, $p_i, q_i \in [1,\infty]$ and $p \in (1,\infty)$ are such that 
	$\frac{1}{p_1} + \frac{1}{p_2} = \frac{1}{p} = \frac{1}{q_1} + \frac{1}{q_2}$.	  
	Then using this inequality, as well as \eqref{EQ28}, we obtain 
	\begin{align}
		\begin{split}
		\Vert (b-I)a\nabla q\Vert_{H^{2}}
		 &\le
		  C(\Vert b-I\Vert_{H^{1.5+\delta}}\Vert a\nabla q\Vert_{H^{2}}
		   +\Vert b-I\Vert_{H^{3}}\Vert a\nabla q\Vert_{H^{0.5}}
		  )  \\&\le
		     \epsilon C(\Vert (a-I)\nabla q\Vert_{H^{2}}+\Vert \nabla q\Vert_{H^{2}})
		      +C_M
		       \le 
		        \epsilon(\epsilon+1)C\Vert \nabla q\Vert_{H^{2}}+C_M
	,\label{EQ58}
	\end{split}
	\end{align}
and 
	\begin{align}
		\Vert (a-I)\nabla q\Vert_{H^{2}}
		 \le 
		  C(\Vert a-I\Vert_{H^{1.5+\delta}}\Vert \nabla q\Vert_{H^{2}}
		   +\Vert a-I\Vert_{H^{3}}\Vert \nabla q\Vert_{H^{0.5}}
		  )  \le 
		     \epsilon C\Vert \nabla q\Vert_{H^{2}} + C_M
		     .\label{EQ59}
	\end{align}
	Combining \eqref{EQ56}--\eqref{EQ59}, we arrive at
	\begin{align}
		\Vert (ba-I)\nabla q\Vert_{H^{2}}
		 \le
		  C\epsilon\Vert \nabla q\Vert_{H^{2}}+C_M
		  ,\llabel{EQ60}
	\end{align}
	and employing the same approach, we conclude that
	\begin{align}
		\Vert (b_{3i}a_{ki}-\delta_{3k})\partial_k q\Vert_{H^{1.5}(\Gamma_{0} \cup \Gamma_{1})}
		 \le
		  \epsilon C\Vert \nabla q\Vert_{H^{2}} +C_M
	.\label{EQ61}
	\end{align}
	Next, we estimate $\fomega$ by writing
	\begin{align}
		\Vert \fomega\Vert_{H^{1}}
		\le
		C_M(\Vert \partial_t b\Vert_{H^{2}}\Vert v\Vert_{H^{2}}
		+\Vert b\Vert_{H^{1.5+\delta}}\Vert v\Vert_{H^{2.5+\delta}}^2
		(\Vert a\Vert_{H^{2.5+\delta}}
		+\Vert J\Vert_{H^{2.5+\delta}}+\Vert \psi_t\Vert_{H^{2.5+\delta}})
		)\le C_M Q^\frac{1}{2}
		.\llabel{EQ62}
	\end{align}
	We note in passing that, 
	to bound the term involving $J^{-1}$, we have used \eqref{EQ28}
	so that
	\begin{align}
		\frac{1}{2} \le J \le \frac{3}{2}.
		\llabel{EQ63}
	\end{align}
Now, it only remains to treat the boundary data $\fgamma$,
	for which we write
	\begin{align}
		\begin{split}
			\Vert \fgamma\Vert_{H^{1.5}(\Gamma_{1})}
			\le&
			C_M(\Vert w\Vert_{H^{5.5}(\Gamma_{1})}
			+\nu\Vert w_t\Vert_{H^{3.5}(\Gamma_{1})}
			+\Vert \partial_t b\Vert_{H^{2}}\Vert v\Vert_{H^{2.5+\delta}}
			\\&+\Vert J\Vert_{H^{2}}\Vert v\Vert_{H^{2}}\Vert b\Vert_{H^{2.5+\delta}}
			(\Vert w_t\Vert_{H^{2.5}(\Gamma_{1})}+\Vert v\Vert_{H^{2}}\Vert b\Vert_{H^{2.5+\delta}})
			)\le C_M Q^\frac{1}{2}.\label{EQ64}
		\end{split}
	\end{align}
	Lastly, 
	we combine \eqref{EQ55}--\eqref{EQ64}
	and use \eqref{EQ23}
	so that we may absorb $\epsilon \Vert q\Vert_{H^{3}}$, establishing~\eqref{EQ54}.	
\end{proof}

\subsection{Tangential estimates}\label{sec.tan}

In this section, we set $\Lambda = (I-~\Delta_\hh)^{1/2}$ and
establish the tangential estimates.

\cole
\begin{Lemma}
\label{L.Tan}
	Let $\nu \in [0,1]$ and 
assume that	$(v,w,q)$
	is a solution of the
	Euler-plate system \eqref{EQ14}--\eqref{EQ21}
	satisfying \eqref{EQ30}
	on $[0,T]$.
	In addition, assume that
	\eqref{EQ25}--\eqref{EQ28}
	hold on $[0,\min\{T_0,T\}]$ 
	for some $M>0$,
	where $\epsilon>0$
	is given by \eqref{EQ23}
	and $T_0>0$ is 
	given by Theorem~\ref{T02}.
	Then, for 
	\begin{align}
		Q(t)
		=
		(\Vert v\Vert_{H^{4}}
		+\Vert w\Vert_{H^{5.5}(\Gamma_{1})}
		+\Vert w_t\Vert_{H^{3.5}(\Gamma_{1})}
		+1)^2
		,
   \llabel{EQ66}
	\end{align}
	we have
	\begin{align}
		\begin{split}
			&
			\Vert  w(t)\Vert_{H^{5.5}(\Gamma_1)}^2
			+       \Vert w_{t}(t)\Vert_{H^{3.5}(\Gamma_1)}^2
			+    \nu \int_{0}^{t}   \Vert  w_{t} \Vert_{H^{4.5}(\Gamma_1)}^2 \, ds
			\\&\indeq
			\le
			C\epsilon \Vert v\Vert_{H^{4}}
			+    \Vert v(t)\Vert_{L^2}^{1/4}
			\Vert v(t)\Vert_{H^{4}}^{7/4}
			+
			Q(0)+C_M\int_{0}^{t} Q(s)\,ds
			,\label{EQ65}	
		\end{split}
		\end{align}
	for $t\in [0,\min\{T_0,T\}]$, 
	where $C>0$ is an absolute constant and $C_M>0$ only depends on $M$
	and in particular, neither on $\nu$ nor on $Q(0)$.
	\end{Lemma}
\colb

Recalling \eqref{EQ14}, \eqref{EQ18}, and
\eqref{EQ30}, we do not have sufficient 
regularity to establish \eqref{EQ65} directly.
As an example, consider $\Lambda^{3.5}w_{tt}$
and $\Lambda^{3.5}w_{t}$ which belong to
$H^{-2}(\Gamma_{1})$ and $L^2$, respectively.
Thus, we will employ difference quotients
 \begin{align}
	\begin{split}
		Df(x) = D_{h,l}f(x)= \frac{1}{h}(f(x+he_l)-f(x))
		\comma l=1,2
	\end{split}
	\label{EQ67}
\end{align}
and
\begin{align}
	\begin{split}
		\tau f(x)
               =\tau_{h,l} f(x)
                = f(x+he_l)
		\comma l=1,2
		,
	\end{split}
	\label{EQ68}
\end{align}
for
$h \in (0,1]$
and $x\in\Omega$.
Occasionally, we shall use the product rule
\begin{align}
	D(fg) = Df \tau g + fDg
	.
	\llabel{EQ69}
\end{align}
We shall frequently utilize the inequality
\begin{align}
   \| D f \|_s \le C\| \nabla_\hh f \|_s,\qquad s\geq 0
   ,
   \llabel{EQ70}	
\end{align}
where $\nabla_{\hh}$ denotes the horizontal gradient.
to estimate the terms with sufficient regularity.
Finally, we note that 
the pressure estimate \eqref{EQ54}
shall be employed 
without mention.

\begin{proof}[Proof of Lemma~\ref{L.Tan}]   
We only present our estimates for the highest order case,
i.e., $5.5$ derivatives for $w$ and so on.	
When we apply $D\Lambda^{2.5}$ to \eqref{EQ18},
all the terms in the resulting equation belong to $H^{-1}(\Gamma_1)$.
Consequently, we may test it with $D\Lambda^{2.5}w_t \in H^1(\Gamma_{1})$,
obtaining
\begin{align}
	\frac{1}{2}\frac{d}{dt}
	 \left(\Vert D\Lambda^{2.5}w_t\Vert_{L^{2}(\Gamma_{1})}^2
	  +\Vert D\Delta_\hh\Lambda^{2.5}w\Vert_{L^{2}(\Gamma_{1})}^2
	  \right)
	   +\nu 
	    \Vert D\nabla_\hh\Lambda^{2.5}\Vert_{L^{2}(\Gamma_{1})}^2
	    =
	     \int_{\Gamma_1} D\Lambda^{2.5}q D\Lambda^{2.5}w_{t} 
	.
	\label{EQ71}
\end{align}
The solutions do not have enough regularity
to estimate the term involving the pressure.
However, we may use \eqref{EQ14} and \eqref{EQ17}
to cancel this term.
To achieve this, we multiply the velocity equation
by $J$, apply
$D\Lambda^{2}$, and test
with~$D\Lambda^{3}v_i$.
The term involving pressure then reads
\begin{align}
	I_q = \int D\Lambda^{2}(b_{ki}\partial_kq)D\Lambda^{3}v_i
	,\label{EQ72}
\end{align}
which after integration by parts in $x_k$ and commuting with $b$ into $D\Lambda^{3}v_i$
results in the boundary term
\begin{align}
	\int_{\Gamma_{1}} D\Lambda^{2.5}qD\Lambda^{2.5}(b_{3i}v_i)
	.\label{EQ73}
\end{align}
Finally, employing the boundary condition \eqref{EQ17} gives us
the desired term that we need to cancel the right-hand side of~\eqref{EQ71}.

In the rest of this section,
we present the rest of the details.
First, we write
\begin{align}
	I_q 
	+
	\int
	 D\Lambda^{2}
	  \left(
	   J\partial_tv_i
	    + v_l b_{jl} \partial_j v_i
	       -\psi_t\partial_3v_i
	   \right)
	    D\Lambda^{3}v_i
	    = 0
	    .
	    \label{EQ74} 
\end{align}
Using the symmetry of $\Lambda$
and the Piola identity, we write
  \begin{align}
	I_q
	 = 
	  \int \partial_k
	   D\Lambda^{2.5}(b_{ki}q)D\Lambda^{2.5}v_i
	   ,\llabel{EQ75}
  \end{align}
from where we integrate by parts and obtain
\begin{align}
	I_q
	 = 
	  -\int 
	    D\Lambda^{2.5}(b_{ki}q)\partial_kD\Lambda^{2.5}v_i
	  +
	   \int_{\Gamma_{1}} 
	    D\Lambda^{2.5}(b_{3i}q)D\Lambda^{2.5}v_i
	  -
	   \int_{\Gamma_{0}}
	    D\Lambda^{2.5}(b_{3i}q)D\Lambda^{2.5}v_i
	  .\label{EQ76} 
\end{align}
The integral over the bottom boundary vanishes thanks
to \eqref{EQ03}$_{3}$ and \eqref{EQ16}, while
for the integral on the top boundary, we write
\begin{align}
	D\Lambda^{2.5}(b_{3i}q)D\Lambda^{2.5}v_i
	 = 
	  D\Lambda^{2.5}qD\Lambda^{2.5}(b_{3i}v_i)
	  + \mathcal{R}_q
	 ,\llabel{EQ77}    
\end{align}
where $\mathcal{R}_q$ is given by
\begin{align}
	\begin{split}
	\mathcal{R}_q
	 &=
	  -\Lambda^{2.5}Dq\Lambda^{2.5}(Db_{3i}v_i)
	  +(\tau b_{3i}\Lambda^{2.5}Dv_i - \Lambda^{2.5}(\tau b_{3i}Dv_i))\Lambda^{2.5}Dq
	  \\&\indeq
	  + (\Lambda^{2.5}(\tau b_{3i}Dq)-\tau b_{3i}\Lambda^{2.5}Dq)\Lambda^{2.5}Dv_i
	  + \Lambda^{2.5}(Db_{3i}q)\Lambda^{2.5}Dv_i
       \\&
	  = \mathcal{R}_{q,1}
	   +\mathcal{R}_{q,2}
	    +\mathcal{R}_{q,3}
	    +\mathcal{R}_{q,4}
	  .\llabel{EQ78}
	  \end{split}
\end{align}
Denoting by $I_{q,\Omega}$ the first integral
on the right-hand side of \eqref{EQ76},
$I_q$ becomes
\begin{align}
	I_q
	 =
	  \int_{\Gamma_{1}} D\Lambda^{2.5}qD\Lambda^{2.5}(b_{3i}v_i)
	   +I_{q,\Omega}
	    +\sum_{l=1}^{4}\int_{\Gamma_{1}} \mathcal{R}_{q,l}
	    .
	    \llabel{EQ79}
\end{align}
Substituting this in~\eqref{EQ74}, summing with $\eqref{EQ71}$,
and employing the boundary condition \eqref{EQ17} yield
\begin{align}
  \begin{split}
    &
	\frac{1}{2}\frac{d}{dt}
	\left(\Vert D\Lambda^{2.5}w_t\Vert_{L^{2}(\Gamma_{1})}^2
	+\Vert D\Delta_\hh\Lambda^{2.5}w\Vert_{L^{2}(\Gamma_{1})}^2
	\right)
	+\nu 
	\Vert D\nabla_\hh\Lambda^{2.5}\Vert_{L^{2}(\Gamma_{1})}^2
      \\&\indeq
	=
	-	\int
	 D\Lambda^{2}
	  \left(
	   J\partial_tv_i
	    + v_l b_{jl} \partial_j v_i
	       -\psi_t\partial_3v_i
	   \right)
	    D\Lambda^{3}v_i
	+I_{q,\Omega}
	+\sum_{l=1}^{4}\int_{\Gamma_{1}} \mathcal{R}_{q,l}
      .
  \end{split}
	\label{EQ80}
\end{align}
Now, it only remains to estimate the terms on the right-hand
side of~\eqref{EQ80}. We start by writing
\begin{align}
	\begin{split}
	\int_{\Gamma_{1}} \mathcal{R}_{q,1}
	 &=
	  \int_{\Gamma_{1}} \Lambda^{1.5}Dq \Lambda^{3.5}(Db_{3i}v_i)
	   \le
	    \Vert \Lambda^{1.5}Dq\Vert_{L^{2}(\Gamma_{1})}
	     \Vert \Lambda^{3.5}(Db_{3i}v_i)\Vert_{L^{2}(\Gamma_{1})}
	      \\&\le
		    C\Vert q\Vert_{H^{3}}	      
	     (\Vert b\Vert_{H^{5}}\Vert v\Vert_{H^{1.5+\delta}}+\Vert b\Vert_{H^{2.5+\delta}}\Vert v\Vert_{H^{4}})
	     \le C_MQ,
	  \label{EQ81}
\end{split}
\end{align}
using the trace and product inequalities.
Next, for the term involving $\mathcal{R}_{q,4}$ we have
\begin{align}
	\begin{split}
	\int_{\Gamma_{1}}\mathcal{R}_{q,4}
	 &\le
	  \Vert \Lambda^{2.5}(Db_{3i}q)\Vert_{L^{2}(\Gamma_{1})}
	   \Vert \Lambda^{2.5}Dv_i\Vert_{L^{2}(\Gamma_{1})}
	   \le
	    C\Vert v\Vert_{H^{4}}(\Vert b\Vert_{H^{4}}\Vert q\Vert_{H^{1.5+\delta}}
	     +\Vert b\Vert_{H^{2.5+\delta}}\Vert q\Vert_{H^{3}})
	      \\&\le C_MQ
	      ,\llabel{EQ82}
	    \end{split}
\end{align}
leaving us with the commutator terms $\mathcal{R}_{q,2}$ and $\mathcal{R}_{q,3}$.
To estimate $\mathcal{R}_{q,3}$ we write
\begin{align}
	\begin{split}
	\int_{\Gamma_{1}}\mathcal{R}_{q,3}
	 &\le
	  \Vert \Lambda^{2.5}(\tau b_{3i}Dq)-\tau b_{3i}\Lambda^{2.5}Dq\Vert_{L^{2}(\Gamma_{1})}
	   \Vert \Lambda^{2.5}Dv_i\Vert_{L^{2}(\Gamma_{1})}
	    \\ &\le
	     C\Vert v\Vert_{H^{4}}
	      (\Vert b\Vert_{H^{4.5+\delta}}\Vert q\Vert_{H^{1.5}}
	       +\Vert b\Vert_{H^{2.5+\delta}}\Vert q\Vert_{H^{3}})
	       \le C_MQ
	       ,\label{EQ83}
	       \end{split}
\end{align}
while to bound $\mathcal{R}_{q,2}$ we first note that
\begin{align}
	\int_{\Gamma_{1}}\mathcal{R}_{q,2}
	 =
	  \int_{\Gamma_{1}}
	   \Lambda^{1.5}Dq
	    \left(\tau b_{3i}\Lambda^{3.5}Dv_i
	     -\Lambda^{3.5}(\tau b_{3i}Dv_i)\right)
	      +\int_{\Gamma_{1}}
	       \Lambda^{1.5}Dq
	        \left(\Lambda(\tau b_{3i}\Lambda^{2.5}Dv_i)
	         -\tau b_{3i}\Lambda^{3.5} Dv_i\right)
	      .\llabel{EQ84}
\end{align}
from where we obtain
\begin{align}
	\begin{split}
	\int_{\Gamma_{1}}\mathcal{R}_{q,2}
	 &\le
	  \Vert \Lambda^{1.5}Dq\Vert_{L^{2}(\Gamma_{1})}
	   \bigl(
	    \Vert \tau b_{3i}\Lambda^{3.5}Dv_i-\Lambda^{3.5}(\tau b_{3i}Dv_i)\Vert_{L^{2}(\Gamma_{1})}
          \\&\indeq\indeq\indeq\indeq\indeq\indeq\indeq\indeq\indeq\indeq
	     +
	      \Vert \Lambda(\tau b_{3i}\Lambda^{2.5}Dv_i)-\tau b_{3i}\Lambda^{3.5} Dv_i\Vert_{L^{2}(\Gamma_{1})}
	       \bigr)
	        \\ &\le
	         C\Vert q\Vert_{H^{3}}
	          (
	           \Vert v\Vert_{H^{4}}\Vert b\Vert_{H^{2.5+\delta}}
	            +
	             \Vert v\Vert_{H^{2.5+\delta}}\Vert b\Vert_{H^{4}}
	              )
	               \le C_MQ
	              .\label{EQ85}
	              \end{split}
\end{align}	

Next, we estimate $I_{q,\Omega}$. 
To do so, we commute $b_{ki}$ into the term involving $v$ and
utilize the divergence-free condition.
First, we note that
\begin{align}
	I_{q,\Omega}
	 = 
	  -\int 
	   D\Lambda^{2}(b_{ki}q)\partial_kD\Lambda^{3}v_i
	  .\llabel{EQ86}	   
\end{align}
We rewrite the integrand as
\begin{align}
	D\Lambda^{2}(b_{ki}q)\partial_kD\Lambda^{3}v_i
	 =
	  \Lambda^{2}Dq\Lambda^{3}(\tau b_{ki} D\partial_kv_i)
	   + \mathcal{R}_{q,\Omega}
	 =
	   -\Lambda^{2}Dq\Lambda^{3}(D b_{ki} \partial_kv_i)
	    + \mathcal{R}_{q,\Omega}
	    ,\llabel{EQ87}
\end{align}
where we have used the divergence-free condition
and the product rule in the second
equality and $\mathcal{R}_{q,\Omega}$ is given by
\begin{align}
	\begin{split}
		\mathcal{R}_{q,\Omega}
		 &=
	     \Lambda^{2}(Db_{ki}q)\Lambda^{3}D\partial_kv_i
	+(\Lambda^{2}(\tau b_{ki}Dq) - \tau b_{ki} \Lambda^{2}Dq)\Lambda^{3}D\partial_kv_i
        \\&\indeq
	+\Lambda^{2}Dq(\tau b_{ki}\Lambda^{3}D\partial_kv_i-\Lambda^{3}(\tau b_{ki} D\partial_kv_i))
       \\&
	=
	 \mathcal{R}_{q,\Omega,1}
	  +\mathcal{R}_{q,\Omega,2}
	   +\mathcal{R}_{q,\Omega,3}
	.\llabel{EQ88}	  
	\end{split}
	\end{align}
Then $I_{q,\Omega}$ becomes
\begin{align}
	I_{q,\Omega}
	 = 
	  I'_{q,\Omega}
	   +\sum_{l=1}^{3}
	    \int \mathcal{R}_{q,\Omega,l}
	    .\llabel{EQ89}
\end{align}
We start by estimating
\begin{align}
  \begin{split}
	I'_{q,\Omega}
	 &\le
	  \Vert \Lambda^{2}Dq\Vert_{L^{2}}\Vert \Lambda^{3}(D b_{ki} \partial_kv_i)\Vert_{L^{2}}
      \\&
	   \le 
	    C\Vert q\Vert_{H^{3}}
	     (\Vert b\Vert_{H^{4}}\Vert v\Vert_{H^{2.5+\delta}}
	      +\Vert b\Vert_{H^{2.5+\delta}}\Vert v\Vert_{H^{4}})
	       \le C_MQ
	       ,
  \end{split}
	       \label{EQ90}
\end{align}
Next, we rewrite the term involving $\mathcal{R}_{q,\Omega,1}$ as
\begin{align}
	\int \mathcal{R}_{q,\Omega,1}
	 =
	  \int   \Lambda^{3}(Db_{ki}q)\Lambda^{2}D\partial_kv_i
	   ,\llabel{EQ91}
\end{align}
from where it follows that
\begin{align}
  \begin{split}
	\int \mathcal{R}_{q,\Omega,l}
	 &\le
	  \Vert \Lambda^{3}(Db_{ki}q)\Vert_{L^{2}}\Vert \Lambda^{2}D\partial_kv_i\Vert_{L^{2}}
    \\&
 	   \le
	    C\Vert v\Vert_{H^{4}}
	     (\Vert b\Vert_{H^{4}}\Vert q\Vert_{H^{1.5+\delta}}
	      +\Vert b\Vert_{H^{2.5+\delta}}\Vert q\Vert_{H^{3}})
	       \le C_MQ
	       .
  \end{split}
	       \label{EQ92}  
\end{align}
Now, we directly estimate $\mathcal{R}_{q,\Omega,3}$ as
\begin{align}
	\begin{split}
	\int \mathcal{R}_{q,\Omega,3}
	 &\le
	  \Vert \Lambda^{2}Dq\Vert_{L^{2}}
	   \Vert \tau b_{ki}\Lambda^{3}D\partial_kv_i-\Lambda^{3}(\tau b_{ki} D\partial_kv_i)\Vert_{L^{2}}
	    \\&\le
	     C\Vert q\Vert_{H^{3}}
	      (\Vert b\Vert_{H^{4}}\Vert v\Vert_{H^{2.5}}
	       +\Vert b\Vert_{H^{2.5+\delta}}\Vert v\Vert_{H^{4}})
	        \le C_MQ
	        ,\label{EQ93}
	        \end{split}
\end{align}
while to treat $\mathcal{R}_{q,\Omega,2}$ we expand $\Lambda^{2}$
as $I-\sum_{l=1}^2 \partial_{mm}$, and integrate by parts in~$m$.
This yields
\begin{align}
	 \begin{split}
	\int \mathcal{R}_{q,\Omega,2}
	 =&
	 \sum_{m=1}^2
	 \left(
	  \int
	   (\Lambda^{2}(\partial_m\tau b_{ki}Dq) - \partial_m\tau b_{ki} \Lambda^{2}Dq)\partial_m\Lambda D\partial_kv_i
	  \right)\\&
	  +\sum_{m=1}^2
	  \left(
	  \int
	  (\Lambda^{2}(\tau b_{ki}\partial_mDq) - \tau b_{ki} \Lambda^{2}\partial_mDq)\partial_m\Lambda D\partial_kv_i
	  \right)\\&+\int
	  (\Lambda^{2}(\tau b_{ki}Dq) - \tau b_{ki} \Lambda^{2}Dq)\Lambda D\partial_kv_i
	   .\label{EQ94}
	  \end{split}
\end{align}
For each integral on the right-hand side of \eqref{EQ94},
we employ the Cauchy-Schwarz inequality, and this results in
\begin{align}
	\int \mathcal{R}_{q,\Omega,2}
	 \le
	 C\Vert v\Vert_{H^{4}}(\Vert b\Vert_{H^{3.5+\delta}}\Vert q\Vert_{H^{2}}+
	  \Vert b\Vert_{H^{4}}\Vert q\Vert_{H^{1.5}}+
	 \Vert b\Vert_{H^{2.5+\delta}}\Vert q\Vert_{H^{3}})
	 \le C_MQ,\label{EQ95}
\end{align}
which concludes the treatment of~$I_{q,\Omega}$.

Going back to \eqref{EQ80}, it only remains to estimate
the first integral on the right-hand side, which
we denote by~$I$ and rewrite as
\begin{align}
	\begin{split}
	I
	 &=
	  \int
	   D\Lambda^{2}
	  (J\partial_tv_i
	  )D\Lambda^{3}v_i
	  +
	   \int
	    D\Lambda^{2}
	     \left(
	     v_l b_{jl} \partial_j v_i
	     \right)D\Lambda^{3}v_i
	   -
	    \int
	    D\Lambda^{2}
	    (\psi_t\partial_3v_i)
	    D\Lambda^{3}v_i
	  \\& = I_1 + I_2 + I_3
	   .\llabel{EQ96}
	   \end{split}
\end{align}
We start with $I_2$ and write
\begin{align}
	I_2 
	 \le
	  C\Vert v\Vert_{H^{4}}
	   (\Vert v\Vert_{H^{3}}\Vert b\Vert_{H^{1.5+\delta}}\Vert v\Vert_{H^{2.5+\delta}}
	     +\Vert v\Vert_{H^{1.5+\delta}}\Vert b\Vert_{H^{3}}\Vert v\Vert_{H^{4}})
	      \le C_MQ,
	      \label{EQ97}
\end{align}
while $I_3$ is bounded as 
\begin{align}
	\begin{split}
	I_3
	 &\le
	  C\Vert v\Vert_{H^{4}}
	   (\Vert \psi_t\Vert_{H^{3}}\Vert v\Vert_{H^{2.5+\delta}}
	     +\Vert \psi_t\Vert_{H^{1.5+\delta}}\Vert v\Vert_{H^{4}})
	  \\&\le
	       C\Vert v\Vert_{H^{4}}
	       (\Vert w_t\Vert_{H^{2.5}(\Gamma_{1})}\Vert v\Vert_{H^{2.5+\delta}}
	       +\Vert w_t\Vert_{H^{1+\delta}(\Gamma_{1})}\Vert v\Vert_{H^{4}})
	       \le C_MQ,
	       \label{EQ98}
	       \end{split}
	       \end{align}
 Next,
to reduce the number of derivatives on $v_t$
we rewrite the integrand in $I_1$ as
 \begin{align}
 	\begin{split}
 	D\Lambda^{2}(J\partial_tv_i)D\Lambda^{3}v_i
 	 &=
 	   JD\Lambda^2\partial_t v_i D\Lambda^{3}v_i+
 	  \Lambda^{2}(DJ \tau \partial_t v_i )D\Lambda^{3}v_i+
 	  (\Lambda^{2}(JD\partial_t v_i)-J\Lambda^{2}D\partial_t v_i) D\Lambda^{3}v_i
 	  \\&= \mathcal{R}_{t,1} + \mathcal{R}_{t,2} + \mathcal{R}_{t,3}
 	  .\llabel{EQ99}
 	  \end{split}
 \end{align}
Using that $\Lambda$ is symmetric, the integral of $\mathcal{R}_{t,1}$
becomes
\begin{align}
	\begin{split}
	\int \mathcal{R}_{t,1}
	 =&
	  \frac{1}{2}\frac{d}{dt} \int J\Lambda^{2}Dv_i \Lambda^{3}Dv_i
	   -\frac{1}{2}\int J_t \Lambda^{2}Dv_i \Lambda^{3} Dv_i
	    \\&-\frac{1}{2}\int \left(\Lambda^{2}(J\Lambda^{2}Dv_i)-J\Lambda^{4}Dv_i\right)\Lambda D\partial_tv_i
	     \\&-\frac{1}{2}\int \left(J\Lambda^{4}Dv_i-\Lambda(J\Lambda^{3}Dv_i)\right)\Lambda D\partial_tv_i
	    = I_{v,t} + I_{v,1} + I_{v,2} + I_{v,3} 
	    ,\label{EQ100}
	    \end{split} 
\end{align}
from where $I_1$ becomes
\begin{align}
	I_1
	 = I_{v,t} + \sum_{l=1}^3 I_{v,l}
	  + \sum_{m=2}^3 \int \mathcal{R}_{t,m}
	  .\llabel{EQ101}
\end{align}
Now, $\mathcal{R}_{t,2}$ and $\mathcal{R}_{t,3}$ are estimated directly as
\begin{align}
	 \int (\mathcal{R}_{t,2} + \mathcal{R}_{t,3})
	 \le
	  C\Vert v\Vert_{H^{4}}
	   \Vert J\Vert_{H^{3}}\Vert v_t\Vert_{H^{2}}
	    \le C_MQ
	    ,\label{EQ102}
\end{align}
and for $I_{v,1}$ we have
\begin{align}
	I_{v,1} \le \Vert v\Vert_{H^{4}}\Vert v\Vert_{H^{3}}\Vert J_t\Vert_{H^{1.5+\delta}}
	 \lec C_MQ.
	 \llabel{EQ103}
\end{align}
Note that to establish $\Vert v_t\Vert_{H^{2}} \le C_MQ^\frac{1}{2}$,
we estimate $v_t$ directly from the velocity equation.
It still remains to estimate the commutator terms $I_{v,2}$
and $I_{v,3}$ which we do by writing
\begin{align}
	I_{v,2}+I_{v,3}
	 \le
	 C\Vert v_t\Vert_{H^{2}}
	  \Vert J\Vert_{H^{3.5+\delta}}\Vert v\Vert_{H^{4}}
	  \le C_MQ.
	  \label{EQ104} 
\end{align}

We now collect our estimates and conclude the proof.
Recalling \eqref{EQ80}, we use 
\eqref{EQ81}--\eqref{EQ83}, \eqref{EQ85}, \eqref{EQ90},
\eqref{EQ92}--\eqref{EQ93}, \eqref{EQ95}, \eqref{EQ97}--\eqref{EQ98}
and \eqref{EQ102}--\eqref{EQ104}, obtaining
\begin{align}
	\frac{1}{2}\frac{d}{dt}
	\left(\Vert D\Lambda^{2.5}w_t\Vert_{L^{2}(\Gamma_{1})}^2
	+\Vert D\Delta_\hh\Lambda^{2.5}w\Vert_{L^{2}(\Gamma_{1})}^2
	\right)
	+\nu 
	\Vert D\nabla_\hh\Lambda^{2.5}\Vert_{L^{2}(\Gamma_{1})}^2
	\le
	I_{v,t} + C_MQ
	.\label{EQ105}
\end{align}
Since $\Lambda^{3}Dv \in C([0,T];L^2)$ and 
$D\Delta_\hh\Lambda^{2.5}w$, $D\Lambda^{2.5}w_t$ belong to $C([0,T];L^2(\Gamma_{1}))$,
we may integrate \eqref{EQ105} in time
and estimate $\int_{0}^{t} I_{v,t}$ as
\begin{align}
	\begin{split}
	\int_{0}^{t} I_{v,t}
	 &\le
	  \int_{0}^{t} (J-I)\Lambda^{2}Dv_i \Lambda^{3}Dv_i\,ds
	  +\int_{0}^{t} \Lambda^{2}Dv_i \Lambda^{3}Dv_i\,ds
	  +     \Vert  v(0) \Vert_{H^{4}}^2
	 \\&\le
	 C\Vert J-I\Vert_{L^{\infty}}\Vert v\Vert_{H^{4}}^2
	 +    \Vert v\Vert_{L^2}^{1/4}
	 \Vert v\Vert_{H^{4}}^{7/4}
	  +     \Vert  v(0) \Vert_{H^{4}}^2
	  ,\llabel{EQ106}
	  \end{split}
\end{align}
from where it follows that
\begin{align}
	\begin{split}
		&
		\Vert  D \Delta_\hh\Lambda^{2.5} w\Vert_{L^2(\Gamma_1)}^2
		+       \Vert D \Lambda^{2.5} w_{t}\Vert_{L^2(\Gamma_1)}^2
		+    \nu \int_{0}^{t}   \Vert  D \nabla_\hh \Lambda^{2.5} w_{t} \Vert_{L^2(\Gamma_1)}^2 \, ds
		\\&\indeq
		\le
	    C\Vert J-I\Vert_{L^{\infty}}\Vert v\Vert_{H^{4}}^2
	    +    \Vert v\Vert_{L^2}^{1/4}
	    \Vert v\Vert_{H^{4}}^{7/4}
	    +	
		\Vert w_{t}(0)\Vert_{H^{3.5}(\Gamma_1)}^2
		+     \Vert  v(0) \Vert_{H^{4}}^2
		+C_M\int_{0}^{t} 
		Q(s)\,ds
		.
	\end{split}
	\llabel{EQ107}
\end{align} 
Recalling \eqref{EQ23} and sending $h \to 0^+$, we conclude
\begin{align}
	\begin{split}
		&
		\Vert  \nabla_\hh\Delta_\hh \Lambda^{2.5} w\Vert_{L^2(\Gamma_1)}^2
		+       \Vert \nabla_\hh \Lambda^{2.5} w_{t}\Vert_{L^2(\Gamma_1)}^2
		+    \nu \int_{0}^{t}   \Vert  \nabla_\hh \nabla_\hh \Lambda^{2.5} w_{t} \Vert_{L^2(\Gamma_1)}^2 \, ds
		\\&\indeq
		\le
		C\epsilon \Vert v\Vert_{H^{4}}^2
		+    \Vert v\Vert_{L^2}^{1/4}
		\Vert v\Vert_{H^{4}}^{7/4}
		+
		\Vert w_{t}(0)\Vert_{H^{3.5}(\Gamma_1)}^2
		+     \Vert  v(0) \Vert_{H^{4}}^2
		+C_M\int_{0}^{t} 
		Q(s)\,ds
		,
	\end{split}
	\llabel{EQ108}
\end{align} 
establishing~\eqref{EQ65}.
\end{proof}

\subsection{The vorticity estimate}\label{sec.vor}

In this section, we establish vorticity 
estimates.
We define the ALE vorticity as $\zeta = \omega \circ \eta$,
where $\omega$ is the Eulerian vorticity.
In coordinates, $\zeta$ reads
\begin{align}
	\zeta_i 
	 = 
	  \epsilon_{ijk}\partial_m v_k a_{mj}
	  ,\label{EQ109}
\end{align}
and it satisfies
\begin{align}
	\begin{split}
		\partial_{t}\zeta_i
		+ v_k a_{jk}\partial_{j} \zeta_i
		- \psi_t a_{j3} \partial_{j} \zeta_i
		= \zeta_k a_{mk}\partial_{m}v_i
		\comma i=1,2,3
		.
	\end{split}
	\llabel{EQ110}
\end{align}
From \eqref{EQ109} it follows that
\begin{align}
	\Vert \zeta\Vert_{H^{1.5+\delta}}^2
	 \le C_M
	 \andand
	  \Vert \zeta\Vert_{H^{3}}^2
	   \le
	    C_MQ
	    .
	    \llabel{EQ111}
\end{align}
Now, we provide a bound for the $H^3$ norm of the vorticity.

\cole
\begin{Lemma}
	\label{L.Vor}
	Let $\nu \in [0,1]$, and
	assume that
	$(v,w,q)$
	is a solution of the
	Euler-plate system \eqref{EQ14}--\eqref{EQ21}
	satisfying \eqref{EQ30}
	on~$[0,T]$.
	In addition, assume that
	\eqref{EQ25}--\eqref{EQ28}
	hold on $[0,\min\{T_0,T\}]$ 
	for some $M>0$,
	where $\epsilon>0$
	is given by \eqref{EQ23}
	and $T_0>0$ is 
	given by Theorem~\ref{T02}.
	Then, for 
	\begin{align}
		Q(t)
		=
		(\Vert v\Vert_{H^{4}}
		+\Vert w\Vert_{H^{5.5}(\Gamma_{1})}
		+\Vert w_t\Vert_{H^{3.5}(\Gamma_{1})}
		+1)^2
		,
   \llabel{EQ113}
	\end{align}
	we have
	\begin{align}
		\Vert \zeta\Vert_{H^{3}}^2
		\le
		\Vert v_0 \Vert_{H^{4}}^2
		+C_M\int_{0}^{t}
		Q(s)\,ds
		\comma
		t\in [0,\min\{T_0,T\}],
		\label{EQ112}
	\end{align}
where $C_M>0$ depends only on $M$
and in particular it does not depend on $\nu$ or~$Q(0)$.
\end{Lemma}
\colb

\begin{proof}[Proof of Lemma~\ref{L.Vor}]
We only present our estimates for the highest order terms, i.e.,
when three derivatives apply to~$\zeta$.
Note that we do not have sufficient regularity to utilize the
energy method for three derivatives of $\zeta$, so we
rely on difference quotients.
We consider \eqref{EQ67} and \eqref{EQ68}
with $l=1,2,3$, and since we need to employ $D$
in the normal direction, we use a total Sobolev extension.
Denote by $\bar{f}$ the continuous extension of $f$
from $H^s$ to $H^s(\Omega_0)$ where $s\in [0,5.5]$ and
\begin{align}
	\Omega_0
	= \mathbb{T}^2 \times \mathbb{R}
	.\llabel{EQ114}
\end{align}
In addition, we require that 
$\bar{f} = 0$ in a neighborhood of 
$\mathbb{T}^2 \times (-1/2,3/2)^\text{c}$.

Thus, denote by $\theta$ the solution of the initial value problem
\begin{align}
	\begin{split}
		\partial_{t}\theta_i
		+ \bar{v}_k \bar{a}_{jk}\partial_{j} \theta_i
		-\bar{\psi}_t \bar{a}_{j3} \partial_{j} \theta_i
		&= \theta_k \bar{a}_{mk}\partial_{m}\bar{v}_i
		\inon{in $\Omega_0$}
		\commaone i=1,2,3
		\\
		\theta(0)&=\bar{\zeta}(0)
		.
	\end{split}
	\label{EQ115}
\end{align}
It follows from~\cite[Lemma 3.6]{KT} that 
\begin{align}
	\zeta = \theta \inon{on $\Omega \times [0,T]$}
	.\llabel{EQ116}
\end{align}
In addition, by performing $L^2$ estimates on \eqref{EQ115}
and utilizing the properties of the extension operator,
we conclude that
\begin{align}
	\theta(x,t) = 0
	\inon{on $(\mathbb{T}^2 \times (-1/2,3/2)^{c}) \times [0,T]$}
	.
	\llabel{EQ117}
\end{align}

Now, let $\alpha=(\alpha_1,\alpha_2,\alpha_3) \in \mathbb{N}^3_0$
be a multi-index with $|\alpha|=2$.
Applying $D\partial^\alpha$ to \eqref{EQ115}
and testing it with $D\partial^\alpha \theta$, we obtain
\begin{align}
	\begin{split}
		\frac12
		\frac{d}{dt}\Vert D\partial^\alpha \theta\Vert_{L^{2}(\Omega_0)}^2
		&=
		-
		\frac12 
		\int_{\Omega_0} g_j\partial_{j} |D\partial^\alpha\theta_i|^2 
		-\int_{\Omega_0} \partial^\alpha (Dg_j \tau \partial_j \theta_i) D \partial^\alpha \theta_i  
		+ \int_{\Omega_0} \partial^\alpha D (\theta_k h_{ki}) D \partial^\alpha \theta_i
		\\&\indeq
		+ \int_{\Omega_0} 
		      \left(g_j \partial^\alpha \partial_j D\theta_i - \partial^\alpha(g_j \partial_j D \theta_i)\right) 
		       D\partial^\alpha \theta_i
		\\&
		= I_1+I_2+I_3+I_4
		,
	\end{split}
	\label{EQ118}
\end{align}
where 
\begin{align}
	g_j= \bar{\psi}_t \bar{a}_{j3} - \bar{v}_k  \bar{a}_{jk}
        \andand
	 h_{ki}=\bar{a}_{mk}\partial_m \bar{v}_i.
	\llabel{EQ119}   
\end{align}
To estimate $I_1$, we integrate by parts in $x_j$, obtaining
\begin{align}
  \begin{split}
	I_1
	 &\le 
	  \Vert \nabla g\Vert_{L^{\infty}(\Omega_0)}\Vert D\partial^\alpha \theta\Vert_{L^{2}(\Omega_0)}^2
	   \le
	    C\Vert \zeta\Vert_{H^{3}}^2
	     \Vert a\Vert_{H^{2.5+\delta}}(\Vert v\Vert_{H^{2.5+\delta}}+\Vert \psi_t\Vert_{H^{2.5+\delta}})
     \\&
	      \le
	       C_M\Vert \zeta\Vert_{H^{3}}^2
	        \le C_MQ
	       .
  \end{split}
	       \label{EQ120}
\end{align}
Next, for $I_2$ we have
\begin{align}
	I_2
	 =
	  \sum_{0\le|\beta|\le |\alpha|}
	   {\alpha \choose \beta}
	    \int_{\Omega_0} \partial^\beta D g_j \partial^{\alpha-\beta} \tau \partial_j \theta_i D \partial^\alpha \theta_i
	     =
	   \sum_{0\le|\beta|\le |\alpha|}
	    {\alpha \choose \beta}
	       \int_{\Omega_0} \mathcal{R}_{2,\beta} \cdot D \partial^\alpha \theta_i
,\label{EQ121}
\end{align}
and we may estimate
\begin{equation}
	\llabel{EQ122}
	\Vert \mathcal{R}_{2,\beta}\Vert_{L^{2}} \le 
	\begin{cases}
		\Vert \partial^\beta D g\Vert_{L^\infty(\Omega_0)}
		 \Vert \partial^{\alpha-\beta} \tau \nabla \theta\Vert_{L^2(\Omega_0)}
		   \le
		    C_M\Vert \zeta\Vert_{H^{3}}
		     \le C_MQ^\frac12, & |{\beta}|=0 \\
		\Vert \partial^\beta D g\Vert_{L^3(\Omega_0)}
		 \Vert \partial^{\alpha-\beta} \tau \nabla \theta\Vert_{L^6(\Omega_0)}
		   \le
		    C_M\Vert \zeta\Vert_{H^{3}}
		     \le C_MQ^\frac12, & |{\beta}|=1 \\
		\Vert \partial^\beta D g\Vert_{L^6(\Omega_0)}
		 \Vert \partial^{\alpha-\beta} \tau \nabla \theta\Vert_{L^3(\Omega_0)}
		   \le  
		    C_MQ^\frac12 , &   |{\beta}|=2.
	\end{cases}
\end{equation}
We proceed to $I_3$ by rewriting it as
\begin{align}
	\begin{split}
	I_3
	&=
	\sum_{0\le|\beta|\le |\alpha|}
	{\alpha \choose \beta}
	\left(\int_{\Omega_0} \partial^\beta D h_{ki} \partial^{\alpha-\beta} \tau \theta_k D \partial^\alpha \theta_i
	 +
	  \int_{\Omega_0} \partial^\beta h_{ki} \partial^{\alpha-\beta} D \theta_k D \partial^\alpha \theta_i
	\right)
	\\&= \sum_{0\le|\beta|\le |\alpha|}
	{\alpha \choose \beta} \int_{\Omega_0}
	 (\mathcal{R}_{31,\beta} + \mathcal{R}_{32,\beta}) \cdot D \partial^\alpha \theta
	 ,
	 \llabel{EQ123}
	 \end{split}
\end{align}
and for $\mathcal{R}_{31,\beta}$ we have
\begin{equation}
	\llabel{EQ124}
	\Vert \mathcal{R}_{31,\beta}\Vert_{L^{2}} \le 
	\begin{cases}
		\Vert \partial^\beta D h\Vert_{L^3(\Omega_0)}
		\Vert \partial^{\alpha-\beta} \tau \theta\Vert_{L^6(\Omega_0)}
		\le
		C_M\Vert \zeta\Vert_{H^{3}}
		 \le C_MQ^\frac12, & |{\beta}|=0 \\
		\Vert \partial^\beta D h\Vert_{L^6(\Omega_0)}
		\Vert \partial^{\alpha-\beta} \tau \theta\Vert_{L^3(\Omega_0)}
		\le
		C_MQ^\frac12, & |{\beta}|=1 \\
		\Vert \partial^\beta D h\Vert_{L^2(\Omega_0)}
		\Vert \partial^{\alpha-\beta} \tau \theta\Vert_{L^\infty(\Omega_0)}
		\le  
		C_MQ^\frac12, &   |{\beta}|=2,
	\end{cases}
\end{equation}
while for $\mathcal{R}_{32,\beta}$ we obtain
\begin{equation}
	\label{EQ125}
	\Vert \mathcal{R}_{32,\beta}\Vert_{L^{2}}  \le 
	\begin{cases}
		\Vert \partial^\beta h\Vert_{L^\infty(\Omega_0)}
		\Vert \partial^{\alpha-\beta} D \theta\Vert_{L^2(\Omega_0)}
		\le
		C_M\Vert \zeta\Vert_{H^{3}}
		 \le C_MQ^\frac12, & |{\beta}|=0 \\
		\Vert \partial^\beta h\Vert_{L^3(\Omega_0)}
		\Vert \partial^{\alpha-\beta} D \theta\Vert_{L^6(\Omega_0)}
		\le
		C_MQ^\frac12, & |{\beta}|=1 \\
		\Vert \partial^\beta h\Vert_{L^3(\Omega_0)}
		\Vert \partial^{\alpha-\beta} D \theta\Vert_{L^3(\Omega_0)}
		\le  
		C_MQ^\frac12, &   |{\beta}|=2.
	\end{cases}
\end{equation}
Using \eqref{EQ121}--\eqref{EQ125}, we may thus write
\begin{align}
	I_2+I_3
	 \le
	 C_MQ
	 .\label{EQ126}
\end{align}
Finally, we expand $I_4$ by writing
\begin{align}
	I_4
	 =
	  -\sum_{1\le|\beta|\le |\alpha|}
	  {\alpha \choose \beta}
	   \int_{\Omega_0} 
	    \partial^\beta g_j \partial^{\alpha-\beta} D \partial_j \theta_i D \partial^\alpha \theta_i
	    =
	     -\sum_{1\le|\beta|\le |\alpha|}
	     {\alpha \choose \beta}
	     \int_{\Omega_0} 
	      \mathcal{R}_{4,\beta} \cdot D \partial^\alpha \theta
	      \llabel{EQ127}
\end{align}
and estimate $\mathcal{R}_{4,\beta}$ as
\begin{equation}
	\label{EQ128}
	\Vert \mathcal{R}_{4,\beta}\Vert_{L^{2}}  \le 
	\begin{cases}
		\Vert \partial^\beta g\Vert_{L^\infty(\Omega_0)}
		\Vert \partial^{\alpha-\beta} D \partial_j \theta\Vert_{L^2(\Omega_0)}
		\le
		C_M\Vert \zeta\Vert_{H^{3}}
		 \le C_MQ^\frac12, & |{\beta}|=1 \\
		\Vert \partial^\beta g\Vert_{L^3(\Omega_0)}
		\Vert \partial^{\alpha-\beta} D \partial_j\theta\Vert_{L^6(\Omega_0)}
		\le
		C_M\Vert \zeta\Vert_{H^{3}}
		 \le C_MQ^\frac12, & |{\beta}|=2.
	\end{cases}
\end{equation}
We now collect \eqref{EQ118}, \eqref{EQ120}, \eqref{EQ126}--\eqref{EQ128},
and then using 
$D\partial^\alpha \theta \in C_TL^2$, we integrate in time, obtaining
\begin{align}
	  \Vert D \partial^\alpha \theta\Vert_{L^{2}(\Omega_0)}^2
	 \le
	  \Vert v_0 \Vert_{H^{4}}^2
	   +\int_{0}^{t}
	      C_MQ\,ds
	  ,\llabel{EQ129}
\end{align}
from where, letting $h \to 0^+$ and recalling that
$\Vert \zeta\Vert_{H^{3}} \le \Vert \theta\Vert_{H^{3}(\Omega_0)}$
we conclude that \eqref{EQ112} holds.
\end{proof}

\subsection{Conclusion of the proof of Proposition~\ref{P01}}\label{sec.con}

Now, we combine the pressure, tangential, and vorticity
estimates and present the proof of Proposition~\ref{P01}.

\begin{proof}[Proof of Proposition~\ref{P01}]
 Let $M>0$ be given, and assume that
 there exists a solution $(v,q,w)$
 satisfying the assumptions in
 Proposition~\ref{P01}. 
 We start by writing the div-curl inequality
\begin{align}
	\Vert v\Vert_{H^{4}}
	 \le
	  C(\Vert \curl v\Vert_{H^{3}} 
	   +\Vert \div v\Vert_{H^{3}} 
	    +\Vert v\cdot N\Vert_{H^{3.5}(\Gamma_{0} \cup \Gamma_{1})}
	     +\Vert v\Vert_{L^{2}}
	  )   ;\label{EQ130}
\end{align}
see~\cite{BB}.
For the divergence term, we use \eqref{EQ14}$_2$ to get
\begin{align}
  \begin{split}
	\Vert \div v\Vert_{H^{3}}^2
	 &=
	  \Vert (a_{ki}-\delta_{ki})\partial_k v_i\Vert_{H^{3}}^2
	   \leq
     \colr
	  C\Vert a_{ki}-\delta_{ki}\Vert_{L^{\infty}}
	  \Vert\partial_k v_i\Vert_{H^{3}}^2
	  +
          C
	  \Vert a_{ki}-\delta_{ki}\Vert_{H^{3}}
          \Vert \partial_k v_i\Vert_{L^{\infty}}^2
    \colb
         \\&
          \leq
	    C\epsilon \Vert v\Vert_{H^{4}}^2
	    +C_M
  \end{split}
	    .\label{EQ131}
\end{align}
Next, for the boundary term in \eqref{EQ130}, we employ
\eqref{EQ16} and \eqref{EQ17} and obtain
\begin{align}
	\begin{split}
	\Vert v \cdot N\Vert_{H^{3.5}(\Gamma_{0} \cup\Gamma_{1})}^2
	 &\le
	 C(\epsilon\Vert v\Vert_{H^{4}}^2
	  +\Vert b-I\Vert_{H^{4}}^2\Vert v\Vert_{H^{1.5+\delta}}^2
	   +\Vert w_t\Vert_{H^{3.5}(\Gamma_{1})}^2
	  ) \\&\le
	       C\epsilon \Vert v\Vert_{H^{4}}^2
	       +\Vert w\Vert_{H^{5.5}(\Gamma_{1})}^2 
	       +\Vert w_t\Vert_{H^{3.5}(\Gamma_{1})}^2
	       +C_M
	      ,\label{EQ132}
	      \end{split}
\end{align}
where we have used 
\begin{align}
	\Vert b-I\Vert_{H^{4}}\Vert v\Vert_{H^{1.5+\delta}}
	\le
	C_M\Vert b-I\Vert_{H^{5}}^\frac{1}{3}\Vert b-I\Vert_{H^{3.5}}^\frac{2}{3} 
	\le
	\Vert b-I\Vert_{H^{5}(\Gamma_{1})} 
	+C_M
	\le
	 \Vert w\Vert_{H^{5.5}(\Gamma_{1})}
	 +C_M
	.\llabel{EQ133}
\end{align}
Finally, for $\curl v$, we recall \eqref{EQ109} so that
\begin{align}
	\begin{split}
	\Vert (\curl v)_i\Vert_{H^{3}}^2
	 &\le
	  \Vert \epsilon_{ijk}\partial_m v_k (a_{mj}-\delta_{mj})\Vert_{H^{3}}^2
	   +\Vert \zeta_i\Vert_{H^{3}}^2
	   \\&\le
	     C\epsilon \Vert v\Vert_{H^{4}}^2 
	       +\Vert v(0)\Vert_{H^{4}}^2
	        +C_M\int_0^t Q\,ds
	        ,\label{EQ134}
	       \end{split}
\end{align}
from where, combining this with \eqref{EQ130}--\eqref{EQ132} and
recalling \eqref{EQ23}, we arrive at
\begin{align}
	\Vert v\Vert_{H^{4}}^2
	 \le
	   C(\Vert v_0\Vert_{H^{4}}^2
	   +
	   \Vert w\Vert_{H^{4.5}(\Gamma_1)}^2
	   +\Vert w_t\Vert_{H^{3.5}(\Gamma_1)}^2
	   )+
	   C_M\left(
	   1
	   + \int_0^t Q\,ds
	   \right)    .\label{EQ135}
\end{align}
Now, we multiply \eqref{EQ135} by $\epsilon_0$, recalling that $\epsilon_0 \leq 1/(2\Cabs)$,  
and add it to the
tangential estimates \eqref{EQ65} obtaining
\begin{align}
	\begin{split}
		&
		\epsilon_0\Vert v\Vert_{H^{4}}^2
		+
		\frac12(\Vert  w\Vert_{H^{5.5}(\Gamma_1)}^2
		+       \Vert w_{t}\Vert_{H^{3.5}(\Gamma_1)}^2
		)+    \nu \int_{0}^{t}   \Vert  w_{t} \Vert_{H^{4.5}(\Gamma_1)}^2 \, ds
		\\&\indeq
		\le
		C\epsilon  \Vert v\Vert_{H^{4}}
		+\Vert v\Vert_{L^2}^{1/4}
		\Vert v\Vert_{H^{4}}^{7/4}
		+
		\Vert w_{t}(0)\Vert_{H^{3.5}(\Gamma_1)}^2
		+     \Vert  v(0) \Vert_{H^{4}}^2
		+
		C_M\left(1
		+    \int_{0}^{t} 
		Q\,ds
		\right).
	\end{split}
	\label{EQ136}
\end{align} 
Finally, recalling \eqref{EQ23} again, we may absorb the factors of $\Vert v\Vert_{H^{4}}$
on the right-hand side and conclude~\eqref{EQ33}.
\end{proof}

\section{Proof of Theorem~\ref{T01}: Uniqueness}\label{sec.un}
In this section, we establish the uniqueness of
solutions constructed in the previous section.
Let $(v,w,q,\eta,a,b,\psi)$ and
$(\tilde{v},\tilde{w},\tilde{q},\tilde{\eta},\tilde{a},\tilde{b},\tilde{\psi})$
be two solutions
of the Euler-plate system
emanating from the same initial data, 
satisfying the compatibility conditions~\eqref{EQ19} and \eqref{EQ20} and achieving the regularity given in \eqref{EQ30}, for some $T>0$.
Denote the difference of these solutions by
\begin{align}
	(V,W,Q,A,B,\Psi)
	 =
	  (v,w,q,a,b,\psi)
	   -
	    (\tilde{v},\tilde{w},\tilde{q},\tilde{\eta},\tilde{a},\tilde{b},\tilde{\psi})
	    .
   \llabel{EQ137}
\end{align}
We aim to show 
\begin{align}
	P(t) \lec
	 \int_0^t P(s) \,ds
	 ,\label{EQ138}
\end{align}
where 
\begin{align}
	P(t)=(\Vert V(t)\Vert_{H^{1.5+\delta}}
	 +\Vert W(t)\Vert_{H^{3+\delta}(\Gamma_1)}
	  +\Vert W_t(t)\Vert_{H^{1+\delta}(\Gamma_{1})})^2
	  \andand
	  P(0)=0
	  \llabel{EQ139}
\end{align}
by establishing the pressure, tangential, and vorticity estimates.
In this section, we employ the symbol ``$\lec $'' to abbreviate ``$\leq C $'', 
where $C>0$ is a constant that can change from line to line.
In doing so, we allow the implicit constants to depend on $T>0$
and the norms \eqref{EQ22} of the solutions. 
Upon taking the difference in \eqref{EQ03}--\eqref{EQ08},
we may conclude that
\begin{align}
	\Vert A\Vert_{H^{2.5+\delta}}
	 +\Vert B\Vert_{H^{2.5+\delta}}
	  +\Vert B_t\Vert_{H^{0.5+\delta}}
	   +\Vert J-\tilde{J}\Vert_{H^{2.5+\delta}}
	    +\Vert \Psi_t\Vert_{H^{1.5+\delta}}
	    \lec
	     \Vert W\Vert_{H^{3+\delta}(\Gamma_1)}
	      +\Vert W_t\Vert_{H^{1+\delta}}
	      .\label{EQ140}
\end{align}

To establish \eqref{EQ138}, we first show that $\Vert Q\Vert_{H^{0.5+\delta}}\lec P^\frac12$.
Note that
$Q$ solves the elliptic problem
\begin{align}
	\begin{split}
		\partial_j(\tilde{b}_{ji}\tilde{a}_{ki}\partial_kQ) 
		&=
		-\partial_j((B_{ji}a_{ki}+\tilde{b}_{ji}A_{ki})\partial_kq)+ 
		 F
		,\inon{in $\Omega$}
		\\
		\tilde{b}_{3i}\tilde{a}_{ki}\partial_k Q &= 
		 - (B_{3i}a_{ki}+\tilde{b}_{3i}A_{ki})\partial_k q
		,\hspace{1.2cm} \inon{on $\Gamma_{0}$}
		\\
		\tilde{b}_{3i}\tilde{a}_{ki}\partial_k Q+ Q &= 
		-(B_{3i}a_{ki}+\tilde{b}_{3i}A_{ki})\partial_k q
		+G
		,\hspace{0.5cm} \inon{on $\Gamma_{1}$}
		.\label{EQ141}
	\end{split}
\end{align}
where we specify $F$ and $G$ next.
First, we write
\begin{align}
	F = F_{B} + F_{V} + F_{A}+ F_{J-\tilde{J}}+ F_{\Psi}
	,\label{EQ142}
\end{align}
with the first two terms in \eqref{EQ142} given by
  \begin{align}
	F_{B}=
	\partial_j(\partial_tB_{ji}v_i)
	 -B_{ji} \partial_{j}(v_m a_{km}) \partial_{k} v_i
	 +B_{ji}\partial_{j}(J^{-1}\psi_t)\partial_{3}v_i
	 +\tilde{v}_m \tilde{a}_{km} \partial_{k}B_{ji} \partial_{j}v_i
	 -\tilde{J}^{-1}\tilde{\psi}_t\partial_{3}B_{ji}\partial_{j}v_i
	\label{EQ143}
   \end{align}
and
\begin{align}
	\begin{split}
      F_{V}=
      &
      \partial_j(\partial_t \tilde{b}_{ji}V_i) 
      -\tilde{b}_{ji} \partial_{j}(V_m a_{km}) \partial_{k} v_i
      -\tilde{b}_{ji} \partial_{j}(\tilde{v}_m \tilde{a}_{km}) \partial_{k} V_i
      +\tilde{b}_{ji}\partial_{j}(\tilde{J}^{-1}\tilde{\psi}_t)\partial_{3}V_i
      \\&
      +V_m a_{km} \partial_{k}b_{ji} \partial_{j}v_i
      +\tilde{v}_m \tilde{a}_{km} \partial_{k}\tilde{b}_{ji} \partial_{j}V_i
      -\tilde{J}^{-1}\tilde{\psi}_t\partial_{3}\tilde{b}_{ji}\partial_{j}V_i
      ,\label{EQ144}	
	\end{split}
\end{align}
while the rest of the terms read
 \begin{align}
 	\begin{split}
       F_{A}=
       \tilde{v}_m A_{km} \partial_{k}b_{ji} \partial_{j}v_i 
      -\tilde{b}_{ji} \partial_{j}(\tilde{v}_m A_{km}) \partial_{k} v_i
     \andand
      F_{\Psi}
      =
      - \tilde{b}_{ji}\partial_{j}(\tilde{J}^{-1}\Psi_t)\partial_{3}v_i
      -\tilde{J}^{-1}\Psi_t\partial_{3}b_{ji}\partial_{j}v_i
      \label{EQ145}
 	\end{split}
 \end{align}
 and
\begin{align}
	\begin{split}
     F_{J-\tilde{J}}
     = \tilde{b}_{ji}\partial_{j}((J^{-1}-\tilde{J}^{-1})(v_3-\psi_t))\partial_{3}v_i
     +(J^{-1}-\tilde{J}^{-1})(v_3-\psi_t)\partial_{3}b_{ji}\partial_{j}v_i
     .\label{EQ146}		
	\end{split}
\end{align}
Next, we write $G$ in \eqref{EQ141}$_3$ as
\begin{align}
	G=
	 G_W+G_V+G_B+G_{J-\tilde{J}}
	 ,\label{EQ147}
\end{align}
where $G_W$ and $G_V$ are specified as
\begin{align}
	\begin{split}
		G_W
		 =
		  \Delta_2^2 W 
		  -  \nu   \Delta_2   W_{t}
		-\tilde{J}^{-1}
		\biggl(
		\sum_{j=1}^{2}  
		\tilde{v}_k \tilde{b}_{jk} \partial_{j}W_t
		+ W_t \partial_{3}b_{3i} v_i
		\biggr)
		\llabel{EQ148}    
	\end{split}
\end{align}
and
\begin{align}
	\begin{split}
		G_V
		 =
		  \partial_{t}\tilde{b}_{3i}V_i
		-\tilde{J}^{-1}
		\biggl(
		\sum_{j=1}^{2}  
		V_k b_{jk} \partial_{j}w_t
		+ \tilde{w}_t \partial_{3}\tilde{b}_{3i} V_i
		- \partial_{j}\tilde{b}_{3i} V_k b_{jk} v_i
		- \partial_{j}\tilde{b}_{3i} \tilde{v}_k \tilde{b}_{jk} V_i
		\biggr)
		    ,\llabel{EQ149}
	\end{split}
\end{align}
while $G_B$ and $G_{J-\tilde J}$ are given by
\begin{align}
	\begin{split}
		G_B=
		 \partial_{t}B_{3i}v_i
		-\tilde{J}^{-1}
		\biggl(
		\sum_{j=1}^{2}  
		\tilde{v}_k B_{jk} \partial_{j}w_t
		+ \tilde{w}_t \partial_{3}B_{3i} v_i
		- \partial_{j}B_{3i} v_k b_{jk} v_i
		- \partial_{j}\tilde{b}_{3i} \tilde{v}_k B_{jk} v_i
		\biggr)
		,\llabel{EQ150} 
	\end{split}
\end{align}
and
\begin{align}
	\begin{split}
		G_{J-\tilde{J}}
		=
		-(J^{-1}-\tilde{J}^{-1})
		\biggl(
		\sum_{j=1}^{2}  v_k b_{jk} \partial_{j}w_t
		+ w_t \partial_{3}b_{3i} v_i
		-       \partial_{j}      b_{3i} v_k b_{jk} v_i
		\biggr)
		.\llabel{EQ151}
	\end{split}
\end{align}
Now, before estimating $Q$ in $H^{0.5+\delta}$,
we rewrite it as $Q=Q_1+Q_2$ where
$Q_1$ solves the elliptic problem
\begin{align}
	\begin{split}
		\partial_j(\tilde{b}_{ji}\tilde{a}_{ki}\partial_kQ_1) 
		&=
		-\partial_j((B_{ji}a_{ki}+\tilde{b}_{ji}A_{ki})\partial_kq)+ 
		F
		,\inon{in $\Omega$}
		\\
		\tilde{b}_{3i}\tilde{a}_{ki}\partial_k Q_1 &= 
		0
		,\hspace{4.5cm} \inon{on $\Gamma_{0}$}
		\\
		\tilde{b}_{3i}\tilde{a}_{ki}\partial_k Q_1+ Q_1 &= 
		0
		,\hspace{4.5cm} \inon{on $\Gamma_{1}$}
		,\llabel{EQ152}
	\end{split}
\end{align}
and $Q_2$ solves
\begin{align}
	\begin{split}
		\partial_j(\tilde{b}_{ji}\tilde{a}_{ki}\partial_kQ_2) 
		&=
		0
		,\hspace{4.4cm} \inon{in $\Omega$}
		\\
		\tilde{b}_{3i}\tilde{a}_{ki}\partial_k Q_2 &= 
		- (B_{3i}a_{ki}+\tilde{b}_{3i}A_{ki})\partial_k q
		,\hspace{1.1cm} \inon{on $\Gamma_{0}$}
		\\
		\tilde{b}_{3i}\tilde{a}_{ki}\partial_k Q_2+ Q_2 &= 
		-(B_{3i}a_{ki}+\tilde{b}_{3i}A_{ki})\partial_k q
		+G
		.\hspace{0.45cm} \inon{on $\Gamma_{1}$}
		.\llabel{EQ153}
	\end{split}
\end{align}
To estimate $Q_1$,
we employ~\cite[Theorem 7.5]{LM}
so that
\begin{align}
  \begin{split}
	\Vert Q_1\Vert_{H^{0.5+\delta}}
	 &\lec
	  \Vert (B_{ji}a_{ki}+\tilde{b}_{ji}A_{ki})\partial_kq\Vert_{H^{0.5+\delta}}
	   +\Vert F\Vert_{H^{-0.5+\delta}}
        \\&
	      \lec 
	       \Vert A\Vert_{H^{1.5+\delta}}+\Vert B\Vert_{H^{1.5+\delta}}+\Vert F\Vert_{H^{-0.5+\delta}}
  .
  \end{split}
  \label{EQ154}  
\end{align}
Now, we estimate $F$ given in~\eqref{EQ142}
by writing
\begin{align}
	\Vert F\Vert_{H^{-0.5+\delta}}
	 \lec
	 \Vert B_t\Vert_{H^{0.5+\delta}}
	  +\Vert B\Vert_{H^{0.5+\delta}}
	  +\Vert V\Vert_{H^{0.5+\delta}}
	  +	\Vert A\Vert_{H^{0.5+\delta}}
	  +\Vert J-\tilde{J}\Vert_{H^{0.5+\delta}}
	  +	\Vert \Psi_t\Vert_{H^{0.5+\delta}},
	  \llabel{EQ155}
\end{align}
and combining this with \eqref{EQ154},
	we conclude that
	\begin{align}
		\Vert Q_1\Vert_{H^{1.5+\delta}}
		 \lec
		  \Vert A\Vert_{H^{1.5+\delta}}+\Vert B\Vert_{H^{1.5+\delta}}
		  +\Vert \partial_t B\Vert_{H^{0.5+\delta}}
		  +\Vert V\Vert_{H^{0.5+\delta}}
	+\Vert J-\tilde{J}\Vert_{H^{0.5+\delta}}
	+	\Vert \Psi_t\Vert_{H^{0.5+\delta}}
	.\label{EQ156}	   
	\end{align}
We proceed to estimate $Q_2$,
employing~\cite[Theorem 6.7]{LM}, so that
 \begin{align}
 	\Vert Q_2\Vert_{H^{0.5+\delta}}
 	 \lec
 	  \Vert B\Vert_{H^{1.5+\delta}}+\Vert A\Vert_{H^{1.5+\delta}}
 	  +\Vert G\Vert_{H^{-1+\delta}(\Gamma_{1})}
 	  .\llabel{EQ157}
 \end{align}
Recalling that $G$ is as in \eqref{EQ147}, we obtain
\begin{align}
	\begin{split}
	\Vert G\Vert_{H^{-1+\delta}(\Gamma_{1})}
	 \lec&
	  \Vert W\Vert_{H^{3+\delta}(\Gamma_{1})}
	   +\Vert W_t\Vert_{H^{1+\delta}(\Gamma_{1})}
	    +\Vert V\Vert_{H^{-0.5+\delta}}
	    +\Vert B_t\Vert_{H^{-0.5+\delta}}
	    \\& +\Vert B\Vert_{H^{0.5+\delta}}
	     +\Vert J-\tilde{J}\Vert_{H^{-0.5+\delta}}
	     ,\llabel{EQ158}
	     \end{split}
\end{align}
from where we conclude that
\begin{align}
	\begin{split}
	\Vert Q_2\Vert_{H^{0.5+\delta}}
	\lec& 
	\Vert B\Vert_{H^{1.5+\delta}}+\Vert A\Vert_{H^{1.5+\delta}}
	+ \Vert W\Vert_{H^{3+\delta}(\Gamma_{1})}
	+\Vert W_t\Vert_{H^{1+\delta}(\Gamma_{1})}
	\\&+\Vert V\Vert_{H^{-0.5+\delta}}
	+\Vert B_t\Vert_{H^{-0.5+\delta}}
	+\Vert J-\tilde{J}\Vert_{H^{-0.5+\delta}}
    .
	\end{split}\label{EQ159}
\end{align}
Since $Q=Q_1+Q_2$, 
combining \eqref{EQ156} with \eqref{EQ159}
and using \eqref{EQ140}, we get
\begin{align}
	\Vert Q\Vert_{H^{0.5+\delta}}
	 \lec
	  \Vert W\Vert_{H^{3+\delta}(\Gamma_{1})}
	  +\Vert W_t\Vert_{H^{1+\delta}(\Gamma_{1})}
	+\Vert V\Vert_{H^{0.5+\delta}}
	\lec P^\frac12
	.\label{EQ160}  
\end{align} 

We now proceed to tangential estimates and aim to show that
\begin{align}
	\begin{split}
		\Vert \Lambda^{1+\delta}W_t\Vert_{L^{2}(\Gamma_{1})}^2
		+\Vert \Delta_\hh\Lambda^{1+\delta}W\Vert_{L^{2}(\Gamma_{1})}^2
		+\nu 
		\int_{0}^{t}\Vert \nabla_\hh\Lambda^{1+\delta}W_t\Vert_{L^{2}(\Gamma_{1})}^2
		\lec
		\Vert V\Vert_{L^{2}}^\frac{1}{1.5+\delta}\Vert V\Vert_{H^{1.5+\delta}}^\frac{2+2\delta}{1.5+\delta}
		+\int_{0}^{t} P.
		\label{EQ161}
	\end{split}
\end{align}
Note that $W$ satisfies
\begin{align}
	W_{tt}+\Delta_\hh^2W -\nu \Delta_\hh W_t = Q
	.\label{EQ162}
\end{align}
We apply $\Lambda^{1+\delta}$ to \eqref{EQ162}
and test it with $\Lambda^{1+\delta}W_t$, obtaining
\begin{align}
	\frac{1}{2}\frac{d}{dt}
	\left(\Vert \Lambda^{1+\delta}W_t\Vert_{L^{2}(\Gamma_{1})}^2
	+\Vert \Delta_\hh\Lambda^{1+\delta}W\Vert_{L^{2}(\Gamma_{1})}^2
	\right)
	+\nu 
	\Vert \nabla_\hh\Lambda^{1+\delta}W_t\Vert_{L^{2}(\Gamma_{1})}^2
	=
	\int_{\Gamma_1} \Lambda^{1+\delta}Q \Lambda^{1+\delta}W_{t} 
	.
	\label{EQ163}
\end{align}
Now, we use the strategy in Section~\ref{sec.tan}
to cancel the pressure term in~\eqref{EQ163}.
We rewrite the difference of the velocity equations as
\begin{align}
	H+b_{ki}\partial_kQ=H_V+H_{J-\tilde{J}}+H_B+H_\Psi+b_{ki}\partial_kQ=0
	,\label{EQ164}
\end{align}
where $H_V$ and $H_B$ are given by
\begin{align}
	H_V=J\partial_{t} V_i
	+ V_1 b_{j1} \partial_{j}v_i
		+ \tilde v_1 \tilde b_{j1} \partial_{j}V_i
		+ V_2 b_{j2} \partial_{j}v_i
		+ \tilde v_2 \tilde b_{j2} \partial_{j}V_i
		+ V_3b_{j3}\partial_{j} v_i
	+ (\tilde v_3-\tilde\psi_t)\tilde b_{j3}\partial_{j} V_i
	\label{EQ165}
\end{align}
and
\begin{align}
	H_B=	 \tilde v_1 B_{j1} \partial_{j}v_i
	+ \tilde v_2 B_{j2} \partial_{j}v_i
	+ (\tilde v_3-\tilde\psi_t)B_{j3}\partial_{j} v_i
	+ B_{ki}\partial_{k}\tilde q
	,\label{EQ166}   	
\end{align}
while
\begin{align}
	H_{J-\tilde{J}}=    (J-\tilde J)\partial_{t} \tilde{v}_i
	\andand
	H_\Psi=-\Psi_tb_{j3}\partial_{j} v_i
	.\label{EQ167}
\end{align}
Also, the divergence-free and boundary conditions become
\begin{align}
	\begin{split}
		b_{ki} \partial_{k}V_i=
		-     B_{ki} \partial_{k} \tilde v_i
		\text{ in }\Omega
		\andand
		 b_{3i}V_i= W_t -B_{3i}\tilde{v}_i
		  \text{ on }\Gamma_{1}
      .
	\end{split}
	\label{EQ168}
\end{align}
We note in passing that using the pressure estimates
\eqref{EQ160} and the equation \eqref{EQ164} for the difference of velocities,
it follows that
\begin{align}
	\Vert V_t\Vert_{H^{-0.5+\delta}}^2
	 \lec P
	 .\llabel{EQ169}
\end{align}
Now, we apply $\Lambda^{0.5+\delta}$ to \eqref{EQ164}
and test it with $\Lambda^{1.5+\delta}V$.
Then, the steps leading to \eqref{EQ80} may be modified
so that we get
\begin{align}
	\frac{1}{2}\frac{d}{dt}
	\left(\Vert \Lambda^{1+\delta}W_t\Vert_{L^{2}(\Gamma_{1})}^2
	+\Vert \Delta_\hh\Lambda^{1+\delta}W\Vert_{L^{2}(\Gamma_{1})}^2
	\right)
	+\nu 
	\Vert \nabla_\hh\Lambda^{1+\delta}W_t\Vert_{L^{2}(\Gamma_{1})}^2
	=
	I_H+I_Q+I_{Q,\Gamma_{1}}
	,
	\label{EQ170}
\end{align}
where $I_H$ and $I_Q$ are given by
\begin{align}
	I_H=
	 \int \Lambda^{0.5+\delta}H\Lambda^{1.5+\delta}V
	 ,\llabel{EQ171}
\end{align}
and 
\begin{align}
	\begin{split}
	I_Q=&
	-\int \left(\Lambda^{1.5+\delta}(b_{ki}Q)-b_{ki}\Lambda^{1.5+\delta}Q\right)\Lambda^{0.5+\delta}\partial_kV_i
	-\int \left(b_{ki}\Lambda^{1.5+\delta}Q-\Lambda (b_{ki}\Lambda^{0.5+\delta}Q)\right)\Lambda^{0.5+\delta}\partial_kV_i
	\\&
	-\int \Lambda^{0.5+\delta}Q\left(b_{ki}\Lambda^{1.5+\delta}\partial_kV_i-\Lambda^{1.5+\delta}(b_{ki}\partial_kV_i)\right)
	+\int \Lambda^{0.5+\delta}Q\Lambda^{1.5+\delta}(B_{ki}\partial_k\tilde{v}_i)
	,\llabel{EQ172}\end{split}
\end{align}
where we used \eqref{EQ168}$_1$,
and $I_{Q,\Gamma_{1}}$ is equal to
\begin{align}
	\begin{split}
	I_{Q,\Gamma_{1}}
	=&\int_{\Gamma_{1}}
	 \left(\Lambda^{1+\delta}(b_{3i}Q)-b_{3i}\Lambda^{1+\delta}Q\right)\Lambda^{1+\delta}V_i
	  +\int_{\Gamma_{1}} \Lambda^{\delta}Q \left(\Lambda(b_{3i}\Lambda^{1+\delta}V_i-b_{3i}\Lambda^{2+\delta}V_i)\right)
	  \\&
	  +\int_{\Gamma_{1}} \Lambda^{\delta}Q \left(b_{3i}\Lambda^{2+\delta}V_i-\Lambda^{2+\delta}(b_{3i}V_i)\right)
	  -\int_{\Gamma_{1}} \Lambda^{1+\delta}Q \Lambda^{1+\delta}(B_{3i}\tilde{v}_i)
	 .\llabel{EQ173}
	 \end{split}
\end{align}
We estimate $I_Q+I_{Q,\Gamma_{1}}$ as
\begin{align}
	I_Q+I_{Q,\Gamma_{1}}
	 \lec
	  \Vert Q\Vert_{H^{0.5+\delta}}
	  (\Vert V\Vert_{H^{1.5+\delta}}
	  +\Vert B\Vert_{H^{2.5+\delta}})
	   \lec
	    \Vert V\Vert_{H^{1.5+\delta}}^2
	     +\Vert W\Vert_{H^{3+\delta}(\Gamma_1)}^2
	      +\Vert W_t\Vert_{H^{1+\delta}(\Gamma_1)}^2
,\label{EQ174}
\end{align}
leaving us with~$I_H$. Recalling \eqref{EQ164}--\eqref{EQ167}, we write
\begin{align}
  \begin{split}
	I_H
	&\lec
	\int \Lambda^{0.5+\delta}(J\partial_tV_i) \Lambda^{1.5+\delta}V
	+
	\int
    (\Vert V\Vert_{H^{1.5+\delta}}
     +\Vert J-\tilde{J}\Vert_{H^{1.5+\delta}}
    \\&\indeq
	     +\Vert B\Vert_{H^{1.5+\delta}}
	      +\Vert \Psi_t\Vert_{H^{1.5+\delta}})
	     \Vert V\Vert_{H^{1.5+\delta}} 
	     .
  \end{split}
	     \llabel{EQ175}
\end{align}
Next, we follow the steps leading to  \eqref{EQ100} so that
\begin{align}
	\begin{split}
	&\int \Lambda^{0.5+\delta}(J\partial_tV_i) \Lambda^{1.5+\delta}V
	\\&\indeq
	 =
	  \frac12 \frac{d}{dt}
	   \int J \Lambda^{0.5+\delta}V \Lambda^{1.5+\delta}V
	    -\frac12 \int J_t \Lambda^{0.5+\delta}V \Lambda^{1.5+\delta}V
	    \\&\indeq\indeq
	    -\frac12 \int \left(\Lambda^{2}(J\Lambda^{0.5+\delta}V)-J\Lambda^{2.5+\delta}V\right)\Lambda^{-0.5+\delta}V_t
	    \\&\indeq\indeq
	    -\frac12 \int \left(J\Lambda^{2.5+\delta}V-\Lambda(J\Lambda^{1.5+\delta}V)\right)\Lambda^{-0.5+\delta}V_t
	    \\&\indeq
	    \lec
	     \frac12 \frac{d}{dt}
	     \int J \Lambda^{0.5+\delta}V \Lambda^{1.5+\delta}V
       \\&\indeq\indeq
	     +\int(\Vert V\Vert_{H^{1.5+\delta}}+\Vert V_t\Vert_{H^{-0.5+\delta}})\Vert V\Vert_{H^{1.5+\delta}}
	    .\label{EQ176}
	    \end{split}
\end{align}
We now combine \eqref{EQ170} with \eqref{EQ174}--\eqref{EQ176} and use \eqref{EQ140}
to obtain
\begin{align}
	\begin{split}
	\frac{1}{2}\frac{d}{dt}
	&\left(\Vert \Lambda^{1+\delta}W_t\Vert_{L^{2}(\Gamma_{1})}^2
	+\Vert \Delta_\hh\Lambda^{1+\delta}W\Vert_{L^{2}(\Gamma_{1})}^2
	\right)
	+\nu 
	\Vert \nabla_\hh\Lambda^{1+\delta}W_t\Vert_{L^{2}(\Gamma_{1})}^2
	\\&\lec
	\frac12 \frac{d}{dt}
	\int J \Lambda^{0.5+\delta}V \Lambda^{1.5+\delta}V
	+
        \int
        (
	\Vert V\Vert_{H^{1.5+\delta}}^2
	+\Vert W\Vert_{H^{3+\delta}(\Gamma_1)}^2
	+\Vert W_t\Vert_{H^{1+\delta}(\Gamma_1)}^2
	)
	,
	\llabel{EQ177}
	\end{split}
\end{align}
from where integrating in time yields~\eqref{EQ161}.

Now, we claim that $\Vert \zeta-\tilde{\zeta}\Vert_{H^{0.5+\delta}} \lec \int_{0}^{t}P$.
To establish this,
as in Section~\ref{sec.vor}, we extend $\zeta$ and $\tilde{\zeta}$
to $\theta$ and $\tilde{\theta}$ that are defined on $\mathbb{T}^2\times \mathbb{R}$.
For functions other than $\zeta$ and $\tilde{\zeta}$,
we do not introduce new notations for extensions, e.g.,
$V$ stands for both $v-\tilde{v}$ and its extension to $\mathbb{T}^2\times \mathbb{R}$.
It follows that $\Theta=\theta-\tilde{\theta}$ solves
\begin{align}
	\partial_t \Theta_i = Z_\Theta + Z_V + Z_A + Z_{\Psi_t} 
	,
	\label{EQ178}
\end{align}
where $Z_\Theta$ and $Z_V$ are given by
\begin{align}
	Z_\Theta
	 = a_{mk}\partial_m v_i \Theta_k
	   -(\tilde{v}_k\tilde{a}_{jk}-\tilde{\psi}_t\tilde{a}_{j3})
	    \partial_j \Theta_i
            \andand
	      Z_V
	      =
	      \tilde{\theta}_k \tilde{a}_{mk} \partial_m V_i
	      -a_{jk}\partial_j \theta_i V_k
     ,
	      \llabel{EQ179}
\end{align} 
and $Z_A$ and $Z_{\Psi_t}$ are equal to
\begin{align}
	Z_A
	 =\tilde{\theta}_k \partial_m v_i A_{mk},
	  -\tilde{v}_k \partial_j \theta_i A_{jk}
	   +\tilde{\psi}_t \partial_j \theta_i A_{j3}
         \andand
	    Z_{\Psi_t} = a_{j3}\partial_j \theta_i \Psi_t
	   .\llabel{EQ180}
\end{align}
Now, we apply 
$\Lambda_3^{0.5+\delta}$ to \eqref{EQ178} and test it with $\Lambda_3^{0.5+\delta}\Theta$
obtaining
\begin{align}
	\frac12 \frac{d}{dt}\Vert \Lambda_3^{0.5+\delta} \Theta\Vert_{L^{2}(\Omega_0)} 
	=\int_{\Omega_0} \Lambda_3^{0.5+\delta}(Z_\Theta + Z_V + Z_A + Z_{\Psi_t})\Lambda_3^{0.5+\delta} \Theta
	,\label{EQ181}
\end{align}
where $\Lambda_3 = (I-\Delta)^\frac12$.
We first note that
\begin{align}
  \begin{split}
     &
	\int_{\Omega_0} \Lambda_3^{0.5+\delta}(Z_V + Z_A + Z_{\Psi_t})\Lambda_3^{0.5+\delta} \Theta
    \\&\indeq
	\lec
	 (\Vert V\Vert_{H^{1.5+\delta}}
	  +\Vert W\Vert_{H^{3+\delta}(\Gamma_{1})}
	   +\Vert W_t\Vert_{H^{1+\delta}(\Gamma_{1})})\Vert \Theta\Vert_{H^{0.5+\delta}(\Omega_0)}
	    .
  \end{split}
    \llabel{EQ182}
\end{align}
Next, we write
\begin{align}
	\begin{split}
     &
	\int_{\Omega_0} \Lambda_3^{0.5+\delta}Z_\Theta \Lambda_3^{0.5+\delta} \Theta
       \\&\indeq
	=
	\int_{\Omega_0} \Lambda_3^{0.5+\delta} (a_{mk}\partial_m v_i \Theta_k)\Lambda_3^{0.5+\delta}\Theta_i
	-\frac12\int_{\Omega_0} 
	 (\tilde{v}_k\tilde{a}_{jk}-\tilde{\psi}_t\tilde{a}_{j3})
	  \partial_j |\Lambda_3^{0.5+\delta}\Theta|^2
	 \\&\indeq\indeq
	 -\int_{\Omega_0} 
	  \biggl(\Lambda_3^{0.5+\delta}
	   \left((\tilde{v}_k\tilde{a}_{jk}-\tilde{\psi}_t\tilde{a}_{j3})\partial_j \Theta_i\right) 
	    - (\tilde{v}_k\tilde{a}_{jk}-\tilde{\psi}_t\tilde{a}_{j3})\Lambda_3^{0.5+\delta}\partial_j \Theta_i\biggr) 
	     \Lambda_3^{0.5+\delta} \Theta_i
	  ,\label{EQ183}
	  \end{split} 
\end{align}
from where we integrate by parts and obtain
\begin{align}
	\int_{\Omega_0} \Lambda_3^{0.5+\delta}Z_\Theta \Lambda_3^{0.5+\delta} \Theta
	 \lec
	  \Vert V\Vert_{H^{1.5+\delta}}^2.
	  \label{EQ184}
\end{align}
Therefore, collecting \eqref{EQ181}--\eqref{EQ184}
gives
\begin{align}
	\Vert \zeta-\tilde{\zeta}\Vert_{H^{0.5+\delta}}
	\lec
	\Vert \Theta\Vert_{H^{0.5+\delta}(\Omega_0)}
	 \lec \int_{0}^{t}P
	 .\llabel{EQ185}
\end{align}

We are now in a position to collect our estimates
and conclude the proof of uniqueness.
As in \eqref{EQ130}, we have
\begin{align}
	\Vert V\Vert_{H^{1.5+\delta}}
	 \lec 
	  \Vert \div V\Vert_{H^{0.5+\delta}}
	   +\Vert \curl V\Vert_{H^{0.5+\delta}}
	    +\Vert V\cdot N\Vert_{H^{1+\delta}(\Gamma_0 \cup \Gamma_1)}
	     +\Vert V\Vert_{L^{2}}
	     .\label{EQ186}
\end{align}
Using \eqref{EQ168}$_1$, we obtain
\begin{align}
	\Vert \div V\Vert_{H^{0.5+\delta}}^2
	 \lec 
	  \Vert (b-I)\nabla V\Vert_{H^{0.5+\delta}}^2
	   +\Vert B\nabla \tilde{v}\Vert_{H^{0.5+\delta}}^2
	    \lec
	     \epsilon \Vert V\Vert_{H^{1.5+\delta}}^2
	     +\Vert W\Vert_{H^{3+\delta}(\Gamma_1)}.\llabel{EQ187}
\end{align}
Next, for the boundary term, we note that $V\cdot N =0$
on the bottom boundary so that \eqref{EQ168}$_2$ implies
\begin{align}
	\begin{split}
	\Vert V\cdot N\Vert_{H^{1+\delta}(\Gamma_{1})}^2
	 &\lec
	  \Vert (I-b)V\Vert_{H^{1+\delta}(\Gamma_{1})}^2
	   +\Vert W_t\Vert_{H^{1+\delta}(\Gamma_{1})}^2
	    +\Vert B\tilde{v}\Vert_{H^{1+\delta}(\Gamma_{1})}^2
	     \\&\lec
	      \epsilon \Vert V\Vert_{H^{1.5+\delta}}^2
	       +\Vert W\Vert_{H^{3+\delta}(\Gamma_1)}
	       +\Vert W_t\Vert_{H^{1+\delta}(\Gamma_{1})}^2
	       .\llabel{EQ188}
	       \end{split}
\end{align}
Finally, for the vorticity term, we have
\begin{align}
	\begin{split}
	\Vert \curl V\Vert_{H^{0.5+\delta}}^2
	 &\lec
	  \Vert \zeta - \tilde{\zeta}\Vert_{H^{0.5+\delta}}^2
	   +\Vert (a-I)\nabla V\Vert_{H^{0.5+\delta}}^2
	    +\Vert \nabla \tilde{v}A\Vert_{H^{0.5+\delta}}^2
	    \\&\lec
	     \epsilon \Vert V\Vert_{H^{1.5+\delta}}
	     +\Vert W_t\Vert_{H^{1+\delta}(\Gamma_{1})}^2
	      +\int_0^t P
	      .\label{EQ189}
	      \end{split}
\end{align}
Now, combining \eqref{EQ186}--\eqref{EQ189}
and absorbing the factors of $\Vert V\Vert_{H^{1.5+\delta}}$,
 it follows that
\begin{align}
	\Vert V\Vert_{H^{1.5+\delta}}^2
	 \lec
	 \Vert V\Vert_{L^{2}}^2	+
	 \Vert W\Vert_{H^{3+\delta}(\Gamma_1)}^2
	 +\Vert W_t\Vert_{H^{1+\delta}(\Gamma_{1})}^2
	 +\int_0^t P
	 .\label{EQ190}  	
\end{align}
We now multiply \eqref{EQ190}
with $\epsilon_0$, add the resulting inequality to~\eqref{EQ161}
and absorb the factors of $\Vert W\Vert_{H^{3+\delta}(\Gamma_{1})}$
and $\Vert W_t\Vert_{H^{1+\delta}(\Gamma_1)}$, which yields
\begin{align}
	P \lec 
	 \Vert V\Vert_{L^{2}}^2+
	 \Vert V\Vert_{L^{2}}^\frac{1}{1.5+\delta}\Vert V\Vert_{H^{1.5+\delta}}^\frac{2+2\delta}{1.5+\delta}
	 +\int_{0}^{t} P
	 .\llabel{EQ191}
\end{align}
Using Young's inequality, we may absorb the factor of $\Vert V\Vert_{H^{1.5+\delta}}$
leaving us with
\begin{align}
	P \lec 
	\Vert V\Vert_{L^{2}}^2
	+\int_{0}^{t} P
	.\label{EQ192}
\end{align}
Finally, dividing \eqref{EQ164} by $J$, testing it with $V$ and
using \eqref{EQ168} implies that
\begin{align}
	\Vert V\Vert_{L^{2}}^2
	 \lec
	  \int_0^t P
	  .\llabel{EQ193}
\end{align}
Together with \eqref{EQ192}, we obtain~\eqref{EQ138}.
Recalling that $P(0)=0$, we conclude the proof
of the uniqueness part of Theorem~\ref{T01}.

\section*{Acknowledgments}
MSA and IK was supported in part by the NSF grant DMS-2205493.

	\ifnum\sketches=1
	\newpage
	\begin{center}
		\bf   Notes ?\rm
	\end{center}
	\huge
	\begin{align}
		& w \in H^{\rr+1.5}(\Gamma_1)   
		\\&
		w_{t} \in H^{\rr-0.5}(\Gamma_1)   
		\\&
		w_{tt} \in H^{\rr-2.5}(\Gamma_1)   
		\\&
		v\in H^{\rr}
		\\&
		v_t\in H^{\rr-2}
		\\&
		q\in H^{\rr-1} 
		\\&
		\eta,\psi\in H^{\rr+2} \text{\ \ \ \ by~$w$}
		\\&
		\eta_t,\psi_t\in H^{\rr}\text{\ \ \ \ by~$w_{t}$}
		\\&
		a,\tda,J\in H^{\rr+1}\text{\ \ \ \ by~$w$}
		\\&
		a_t \in H^{\rr-1}\text{\ \ \ \ by~$w$ and $w_{t}$}
		\\&
		\tda_t,J_t\in H^{\rr-1}\text{\ \ \ \ by~$w_{t}$}
		\\&
		\zeta\in H^{\rr-1}
	\end{align}
	
	\newpage
	\huge
	\begin{align}
		& W \in H^{\rr+0.5}(\Gamma_1)   
		\\&
		W_{t} \in H^{\rr-1.5}(\Gamma_1)   
		\\&
		W_{tt} \in H^{-\rr-1.5}(\Gamma_1)   
		\\&
		V\in H^{\rr-1}
		\\&
		V_t\in H^{\rr-3}
		\\&
		Q\in H^{\rr-2} 
		\\&
		E,\Psi\in H^{\rr+1} \text{\ \ \ \ by~$W$}
		\\&
		E_t, \Psi_t\in H^{\rr-1}\text{\ \ \ \ by~$W_{t}$}
		\\&
		A,B\in H^{\rr}\text{\ \ \ \ by~$W$}
		\\&
		A_t \in H^{\rr-2}\text{\ \ \ \ by~$W$ and $W_{t}$}
		\\&
		B_t\in H^{\rr-2}\text{\ \ \ \ by~$W_{t}$}
		\\&
		\Theta\in H^{\rr-2}
		\\&
		\Theta_t\in H^{\rr-3}
		\\&
		a b - \tilde a \tilde b
		= A b + \tilde a B
		\\&
		a b - \tilde a \tilde b
		= a B + A \tilde b
		\\&
		a b c - \tilde a \tilde b \tilde c
		= A b c + \tilde a B c + \tilde a \tilde b C
	\end{align}
	
	\fi
	
\end{document}